\documentclass[final]{svjour3}
\usepackage{amsmath,amssymb,enumitem,times}
\usepackage[letterpaper,margin=1in]{geometry}
\usepackage{color}

\newenvironment{enumerateII}
{\setlength{\leftmargini}{2.5em}
\begin{enumerate}[align=left]
  \setlength{\labelwidth}{1.5em}
  \setlength{\labelsep}{1.0em} 
  \setlength{\itemsep}{-0pt}

  }
{\end{enumerate}}

\def\R{\mathbb{R}}

\def\N{\mathbb{N}}
\def\P{\mathbb{P}}
\def\E{\mathbb{E}}

\newcommand{\be}{\begin{equation}}
\newcommand{\ee}{\end{equation}}
\newcommand{\bea}{\begin{eqnarray}}
\newcommand{\eea}{\end{eqnarray}}
\newcommand{\beann}{\begin{eqnarray*}}
\newcommand{\eeann}{\end{eqnarray*}}
\newcommand{\benn}{\begin{equation*}}
\newcommand{\eenn}{\end{equation*}}

\def\ra{\rightarrow}
\def\I{\infty}

\newcommand{\cB}{{\mathcal B}}  
\newcommand{\cD}{{\mathcal D}}  
\newcommand{\cF}{{\mathcal F}}  
\newcommand{\cL}{{\mathcal L}}  
\newcommand{\cO}{{\mathcal O}}  
\newcommand{\cW}{{\mathcal W}}  
\newcommand{\cX}{{\mathcal X}}  

\newcommand{\e}{\textnormal e}  
\newcommand{\opnorm}{\ensuremath{|\hspace{-0.75pt}|\hspace{-0.75pt}|}} 

\newcommand{\dom}{\mathcal{B}}

\spnewtheorem{example2}{Example}{\bf}{\it}

\numberwithin{theorem}{section}
\numberwithin{corollary}{section}
\numberwithin{remark}{section}
\numberwithin{proposition}{section}
\numberwithin{lemma}{section}
\numberwithin{example2}{section}

\title{Large Deviations for Nonlocal Stochastic Neural Fields
}
\titlerunning{Large Deviations for Nonlocal Stochastic Neural Fields}        

\author{Christian Kuehn \and Martin G.~Riedler}

\institute{C.~Kuehn \at
              Vienna University of Technology \\
	      	  Institute for Analysis and Scientific Computing\\
              1040 Vienna, Austria\\
              \email{ck274@cornell.edu}           
           \and
              M.~G.~Riedler \at
              Johannes Kepler Universit\"at \\
	      	  Institute for Stochastics\\
		      Altenbergerstra{\ss}e 69, 1040 Linz, Austria\\
              \email{martin.riedler@jku.at} 
}

\date{Received: date / Accepted: date}

\begin{document}

\maketitle

\begin{abstract}
We study the effect of additive noise on integro-differential neural field equations. In particular, we analyze an Amari-type model driven by a $Q$-Wiener process and focus on noise-induced transitions and escape. We argue that proving a sharp Kramers' law for neural fields poses substantial difficulties but that one may transfer techniques from stochastic partial differential equations to establish a large deviation principle (LDP). Then we demonstrate that an efficient finite-dimensional approximation of the stochastic neural field equation can be achieved using a Galerkin method and that the resulting finite-dimensional rate function for the LDP can have a multi-scale structure in certain cases. These results 
form the starting point for an efficient practical computation of the LDP. Our approach also provides the technical basis for further rigorous study of noise-induced transitions in neural fields based on Galerkin approximations.
\end{abstract}

\section{Introduction}  
\label{sec:intro}

Starting from the classical works of Wilson/Cowan \cite{WilsonCowan} and Amari \cite{Amari} there has
been considerable interest in the analysis of spatio-temporal dynamics of meso-scale models of neural
activity. Continuum models for neural fields often take the form of nonlinear integro-differential 
equations where the integral term can be viewed as a nonlocal interaction term; see \cite{GerstnerKistler}
for a derivation of neural field models. Stationary states, travelling waves and pattern formation for neural fields
have been studied extensively see, {e.g.}, \cite{Coombes,ErmentroutMcLeod} or the recent review by Bressloff 
\cite{BressloffReview} and references therein. 

In this paper we are going to study a stochastic neural field model. There are several motivations for
our approach. In general, it is well-known that intra- and inter-neuron \cite{DestexheRudolph-Lilith} 
dynamics are subject to fluctuations. Many meso- or macro-scale continuum models
have stochastic perturbations due to finite-size effects \cite{GinzburgSompolinsky,SoulaChow}. 
Therefore, there is certainly a genuine need to develop new techniques to analyze random neural systems
\cite{LaingLord}. For stochastic neural fields there is also the direct motivation to understand
the relation between noise and short-term working memory \cite{LaingTroyGutkinErmentrout} as well as 
noise-induced phenomena \cite{Moreno-BoteRinzelRubin} in perceptual bistability \cite{vanEe}. Although 
an eventual goal is to match results from stochastic neural fields to actual cortex data \cite{FunahashiBruceGoldman-Rakic}
we shall not attempt such a comparison here. However, the techniques we develop could have
 the potential to make it easier to understand the relation between models and experiments; see Section 
 \ref{sec:discussion} for a more detailed discussion.

There is a relatively small amount of fairly recent work on stochastic neural fields 
which we briefly review here. Brackley and Turner \cite{BrackleyTurner} study a neural field with 
a gain function which has a random firing threshold. Fluctuating gain functions 
are also considered by Coombes et {al.}~\cite{Coombesetal}. Bressloff and Webber \cite{BressloffWebber} 
analyze a stochastic neural field equation with multiplicative noise while Bressloff and Wilkinson 
\cite{BressloffWilkerson} study the influence of extrinsic noise on neural fields. In all these
works the focus is on the statistics of travelling waves such as front diffusion and the effects
of noise on the wave speed. Hutt et {al.}~\cite{HuttLongtinSchimansky-Geier} study the influence
of external fluctuations on Turing bifurcation in neural fields. Kilpatrick and Ermentrout 
\cite{KilpatrickErmentrout} are interested in stationary bump solutions. They observe 
numerically a noise-induced passage to extinction as well as noise-induced switching of bump 
solutions and conjecture that ``a Kramers' escape rate calculation'' \cite[p.16]{KilpatrickErmentrout}
could be applied to stochastic neural fields but they do not carry out this calculation.  In particular,
the question is whether one can give a precise estimate of the mean transition time between metastable
states for stochastic neural field equations; for a precise statement of the classical Kramers' law see Section 
\ref{sec:LDP}, equation \eqref{eq:Kramers}. However, to the best of our knowledge there seems to be no 
general Kramers' law or large deviation principle (LDP) calculation available for continuum neural 
field models although large deviations have been of recent interest in neuroscience applications 
\cite{Bressloff2,FaugerasMacLaurin}. It is one of the main goals of this paper to provide the basic 
steps towards a general theory.

Although Kramers' law \cite{Berglund3} and LDPs \cite{FreidlinWentzell,DemboZeitouni} are 
well understood for finite-dimensional stochastic differential equations (SDEs), the work for
infinite-dimensional evolution equations is much more recent. In particular, it has been
shown very recently that one may extend Kramers' law to certain stochastic partial differential
equations (SPDEs) \cite{Barret,BerglundGentz11,BerglundGentz10} driven by space-time white
noise. The work of Berglund and Gentz \cite{BerglundGentz10} provides a quite general strategy 
how to `lift' a finite-dimensional Kramers' law to the SPDE setting using a Galerkin approximation
due to Bl\"{o}mker and Jentzen \cite{BloemkerJentzen}. Since the transfer of PDE techniques 
to neural fields has been very successful, either directly \cite{LaingTroy} or indirectly 
\cite{CoombesOwen,BressloffReview}, one may conjecture that the same strategy also works 
for SPDEs and stochastic neural fields.\medskip

In this paper we consider a rate-based (or Amari) neural field model driven by a $Q$-Wiener 
process $W$ 
\be
\label{eq:Amari1_intro}
dU_t(x)=\left[-\alpha U_t(x)+\int_{\dom} w(x,y)f(U_t(y))dy\right]dt+\epsilon dW_t(x)
\ee
for a trace-class operator $Q$, nonlinear gain function $f$ and an interaction kernel $w$; 
the technical details and definitions are provided in Section \ref{sec:Amari}. Observe 
that \eqref{eq:Amari1_intro} is a relatively general formulation
of a nonlocal neural field. Hence we expect that the techniques developed in 
this paper carry over to much wider classes of neural fields 
beyond \eqref{eq:Amari1_intro} such as activity-based models. 

\begin{remark}
To avoid confusion we alert readers familiar with neural fields 
that the nonlinear gain function $f$ in \eqref{eq:Amari1_intro} is sometimes also
called a `rate function'. However, we reserve `rate function' for
a functional, to be denoted later by $I$, arising in the context of an LDP as this 
convention is standard in the context of LDPs.
\end{remark}

Our main goal in the study of \eqref{eq:Amari1_intro} is to provide estimates on the 
mean first passage times between metastable states. In particular, we develop the basic 
analytical tools to approximate equation \eqref{eq:Amari1_intro} as well as its rate
function using a finite-dimensional Galerkin approximation. By making the rate function
as explicit as possible we do not only provide a starting point for further analytical 
work but also provide a framework for efficient numerical methods to analyze
metastable states.\medskip

The paper is structured as follows: The motivation 
for \eqref{eq:Amari1_intro} is given in Section \ref{sec:gain_perturb} where a formal 
calculation shows that a space-time white noise perturbation of the gain function 
in a deterministic neural field leads to \eqref{eq:Amari1_intro}. In Section \ref{sec:deterministic} 
we briefly describe important features of the deterministic
dynamics for \eqref{eq:Amari1_intro} where $\epsilon=0$. In particular, we collect 
several examples from the literature where the classical Kramers' stability configuration
of bistable stationary states separated by an unstable state occurs for Amari-type
neural fields. In Section \ref{sec:LDP} we introduce the notation for Kramers' law
and LDPs and state the main theorem on finite-dimensional rate functions. In 
Section \ref{sec:gradient} we argue that a direct approach to Kramers' law via 
`lifting' for \eqref{eq:Amari1_intro} is likely to fail. Although the Amari model has a 
hidden energy-type structure we have not been able to generalize the gradient-structure 
approach for SPDEs to the stochastic Amari model. This raises doubt 
whether a Kramers' escape rate calculation can actually be carried out {i.e.}~whether 
one may express the pre-factor of the mean first-passage in the bistable 
case explicitly. Based on these considerations we restrict ourselves to just derive 
an LDP. In Section \ref{sec:direct} the LDP is established by a direct transfer of a 
result known for SPDEs. The disadvantage of this approach is that the resulting rate 
function is difficult to calculate, analytically or numerically, in practice. Therefore,
we establish in Section \ref{sec:Galerkin} the convergence of a suitable Galerkin
approximation for \eqref{eq:Amari1_intro}. Using this approximation one may apply
results about the LDP for SDEs which we carry out in Section \ref{sec:approximation}.
In this context, we also notice that the trace-class noise can induce a multi-scale 
structure of the rate function in certain cases. The 
last two observations lead to a tractable finite-dimensional approximation of the 
LDP and hence also an associated finite-dimensional approximation for first-exit time
problems. We conclude the paper in Section \ref{sec:discussion}
with implications of our work and remarks about future problems.

\section{Amari-Type Models}
\label{sec:Amari}

In this study we consider stochastic neural field models with additive noise of the form 
\be
\label{eq:Amari1}
dU_t(x)=\left[-\alpha U_t(x)+\int_{\dom} w(x,y)f(U_t(y))dy\right]dt+\epsilon dW_t(x)
\ee
for $x\in\dom\subseteq \R^d$, a small parameter $\epsilon>0$,  and $t\geq 0$, where $\dom$ 
is bounded and closed. In
\eqref{eq:Amari1} the solution $U$ models the averaged electrical potential generated
by neurons at location $x$ in an area of the brain $\dom$. Neural field equations of
the form \eqref{eq:Amari1} are called Amari-type equations or a rate-based neural
field models. The equation is driven by an adapted space-time stochastic process $W_t(x)$
on a filtered probability space $(\Omega,\cF,(\cF_t)_{t\geq 0},\P)$. The precise
definition of the process $W$ will be given below. 

The parameter $\alpha>0$ is the decay rate for the potential, $w:\dom\times\dom\to\R$
is a kernel that models the connectivity of neurons at location $x$ to neurons at
location $y$. Positive values of $w$ model excitatory connections and negative values
model inhibitory connections. The gain function $f:\R\to\R_+$ relates the potential
of neurons to inputs into other neurons. Typically, the gain functions are chosen
sigmoidal, for  example, (up to affine transformations of the argument)
$f(u)=(1+\e^{-u})^{-1}$ or $f(u)=(\tanh(u)+1)/2$. These examples of gain functions are
bounded, infinitely often differentiable with bounded derivatives. However, throughout
the paper we only make the standing assumption that 
\begin{enumerateII}
 \item[(H1)] the gain function $f$ is globally Lipschitz continuous on $\R$.
\end{enumerateII}
We may transfer equation \eqref{eq:Amari1} into the Hilbert space setting
of infinite-dimensional stochastic evolution equations \cite{daPratoZabczyk,PrevotRoeckner}
for the Hilbert space $L^2(\dom)$. Subsequently brackets $\langle\cdot,\cdot\rangle$ always denote
the inner product on this Hilbert space. Moreover, we introduce the following notation.
Firstly, $F$ denotes the nonlinear Nemytzkii-operator defined from $f$, i.e., $F(g)(x)=f(g(x))$
for any function $g\in L^2(\dom)$. The condition (H1) implies that $F:L^2(\dom)\ra L^2(\dom)$ 
is a Lipschitz continuous operator. Often, spatially continuous solutions to \eqref{eq:Amari1} are also of interest and thus we note that the Nemytzkii-operator also preserves its Lipschitz continuity
on the Banach space $C(\dom)$ with its norm $\|g\|_0=\sup_{x\in\dom}|g(x)|$ due to $\dom$ being bounded.\footnote{We note that the boundedness assumption on the domain $\dom$ in this study is only necessary when dealing with results in the space $C(\dom)$ as is the appropriate space for the LDP results. All other results in this paper which only deal with the space $L^2(\dom)$, e.g., existence of solutions and convergence of the Galerkin approximation, are also valid for unbounded spatial domains.} Secondly, the 
linear operator $K$ is the integral operator defined by the kernel $w$
\be
\label{eq:integral_operator}
Kg(x)=\int_\dom w(x,y) g(y)\,dy \qquad\forall\,g\in L^2(\dom).
\ee
Throughout the paper we assume that
\begin{enumerateII}
 \item[(H2)] the kernel $w$ is such that $K$ is a compact, self-adjoint operator on $L^2(\dom)$.
\end{enumerateII}
We note that an integral operator is self-adjoint if and only if the kernel is symmetric, i.e.,
$w(x,y)=w(y,x)$ for all $x,y\in\dom$. A sufficient condition for the compactness of $K$ is, 
e.g., $\|w\|_{L^2(\dom\times\dom)}<\I$ in which case the operator is called a Hilbert-Schmidt 
operator. Since $\dom$ is bounded the continuity of the kernel $w$ on 
$\dom\times\dom$ implies the compactness of $K$ considered an integral operator on $C(\dom)$.\medskip

Then we re-write equation \eqref{eq:Amari1} as an Hilbert space-valued stochastic evolution equation
\be
\label{eq:Amari2}
dU_t=\left[-\alpha U_t+KF(U_t)\right]dt+\epsilon dW_t\,,
\ee
where $W$ is an $L^2(\dom)$-valued stochastic process. Interpreting the original equation in 
this form we now give a definition of the noise process assuming that
\begin{enumerateII}
 \item[(H3)] $W$ is a $Q$-Wiener process on $L^2(\dom)$, where the covariance operator $Q$ 
is a  non-negative, symmetric trace class operator on $L^2(\dom)$.
\end{enumerateII}
For a detailed explanation of a Hilbert space-valued $Q$-Wiener process and its covariance operator
we refer to, e.g., \cite{daPratoZabczyk,PrevotRoeckner}. As the operator $Q$ is non-negative, symmetric
and of trace class there exists an orthonormal basis of $L^2(\dom)$ consisting of eigenfunctions $v_i$
and corresponding non-negative real eigenvalues $\lambda_i^2$ which satisfy 
$\sum_{i=1}^\I \lambda_i^2<\infty$ . It then holds that the $Q$-Wiener process $W$ satisfies
\be
\label{eq:wiener_process}
W_t=\sum_{i=1}^{\I} \lambda_i\,\beta_t^i\,v_i\,,
\ee
where $\beta^i$ are a sequence of independent scalar Wiener processes 
(cf.~\cite[Prop 2.1.10]{PrevotRoeckner}). The series \eqref{eq:wiener_process} converges 
in the mean-square on $C([0,T],L^2(\dom))$. Furthermore, a straightforward adaptation of 
the proof of \cite[Prop 2.1.10]{PrevotRoeckner} shows that convergence in the mean-square 
also holds in the space $C([0,T],C(\dom))$ for every $T>0$ if $v_i\in C(\dom)$ for all 
$i$ (corresponding to non-zero eigenvalues) and 
$\sup_{x\in\dom}\left|\sum_{i=1}^\I \lambda_i^2v_i(x)^2\right|<\infty$.

The existence and uniqueness of mild solutions to \eqref{eq:Amari2} with trace class
noise for given initial condition $U_0\in L^2(\dom)$ is guaranteed under the Lipschitz
condition on $f$, cf.~\cite{daPratoZabczyk}, and we can write the solution
in its mild form
\be
\label{eq:mild_solution}
U_t=\e^{-\alpha t} U_0+\int_0^t \e^{-\alpha(t-s)}\,KF(U_s)\,ds +\int_0^t \e^{-\alpha(t-s)}\,dW_s.
\ee
The solution possesses a modification in $C([0,T],L^2(\dom))$ and from now on we always
identify the solution \eqref{eq:mild_solution} with its continuous modification. It is
worthwhile to note that for cylindrical Wiener processes -- and thus in particular
space-time white noise -- there does not exist a solution to \eqref{eq:Amari2}. This 
contrasts with other well-studied infinite-dimensional stochastic  evolution equations, e.g., 
the stochastic heat equation. Due to the representation of the solution \eqref{eq:mild_solution} 
it follows that a solution can only be as spatially regular as the stochastic convolution 
$\int_0^t \e^{-\alpha(t-s)}\,dW_s$. In the present case the semigroup generated by the 
linear operator is not smoothing in contrast to, e.g., the semigroup generated by the 
Laplacian in the heat equation. Thus the stochastic convolution is only as smooth as the 
noise which for space-time white noise is not even a well-defined function. To be more specific, for 
cylindrical Wiener noise the series representation  of the stochastic convolution (cf.~see equation \eqref{eq:stochastic_convolution} below) does not converge in a suitable probabilistic sense.

We next aim to strengthen the spatial regularity of the solution \eqref{eq:mild_solution}
which will be required later on. According to \cite[Thm.~7.10]{daPratoZabczyk} the 
solution \eqref{eq:mild_solution} is a continuous process taking values in the Banach 
space $C(\dom)$ if the initial condition satisfies $u_0\in C(\dom)$, the linear part 
in the drift of \eqref{eq:Amari2} generates a strongly continuous semigroup on $C(\dom)$, 
the non-linear term $KF$ is globally Lipschitz continuous on $C(\dom)$ and, finally, if 
the stochastic convolution is a continuous process taking values in $C(\dom)$. It is 
easily seen, that the first conditions are satisfied and sufficient conditions for the 
latter property are given in the following lemma. 

\begin{lemma}
\label{lemma:spatial_continuity} 
Assume that the orthonormal basis functions $v_i$ are Lipschitz continuous with 
Lipschitz constants $L_i$ such that
\be
\label{lemma:spatial_continuity_conditions}
\sup_{x\in\dom}\left|\sum_{i=1}^\I \lambda_i^2\,v_i(x)^2\right|<
\infty\,,\qquad\sup_{x\in\dom}\left|\sum_{i=1}^\I \lambda_i^2L_i^{2\rho}|v_i(x)|^{2(1-\rho)}\right|\, <\,\infty
\ee
for a $\rho\in(0,1)$. Then the process
\be
\label{eq:stochastic_convolution}
O(x,t):=\int_0^t \e^{-\alpha(t-s)}dW_s(x)=
\sum_{i=1}^\I \lambda_i \int_0^t \e^{-\alpha(t-s)}\,d\beta^i_s\,v_i(x)
\ee
possesses a modification with $\gamma$-H\"older continuous paths 
in $\R_+\times\dom$ for all $\gamma\in(0,\rho/2)$.
\end{lemma}

\begin{proof} We prove the lemma applying the Kolmogorov-Centsov Theorem (cf.~\cite[Thm.~3.3 \& Thm.~3.4]{daPratoZabczyk}). 
Throughout the proof $C$ is some finite, constant which may change from line 
to line but is independent of $x,y\in\dom$ and $t,s\geq 0$. We start showing 
that the process $O$ is H\"older continuous in the mean-square in each direction. 
As $v_i$ are assumed continuous these are pointwise uniquely given and each 
$O(t,x)$ is for fixed $x\in\dom$ and $t\geq 0$ a Gaussian random variable due 
to $\sum_{i=1}^\I \lambda^2_i v_i(x)^2<\infty$. Hence, for all 
$0\leq s\leq t$ and all $x,y\in\dom$ we obtain 
\beann
\E|O(x,t)-O(y,t)|^2 &=&
\sum_{i=1}^\I \lambda_i^2\int_0^t\e^{-2\alpha(t-s)}\,ds\,|v_i(x)-v_i(y)|^2\\[1ex]
&\leq& C\,\sup_{z\in\dom}\left|\sum_{i=1}^\I \lambda_i^2L_i^{2\rho}|v_i(z)|^{2(1-\rho)}\right| |x-y|^\rho
\eeann 
using
\benn
|v_i(x)-v_i(y)|^2\leq L_i^{2\rho}|x-y|^\rho|x-y|^\rho\bigl(|v_i(x)|+|v_i(y)|\bigr)^{2(1-\rho)}
\eenn
for every $\rho\in[0,1]$ and $|x-y|^\rho\leq \textnormal{diam}(\dom)^\rho$. 
Next, for the temporal regularity we 
obtain
\beann
\lefteqn{\E|O(x,t)-O(x,s)|^2}\\[1ex]
&\!\!\!\!=&\!\!\!\sum_{i=1}^\I \lambda_i^2\, v_i(x)^2 \int_s^t \e^{-2\alpha(t-r)}dr +
 \sum_{i=1}^\I \lambda_i^2\,v_i(x)^2\int_0^s\left|\e^{-\alpha(t-r)}-\e^{-\alpha(s-r)}\right|^2dr\\[1ex]
&\!\!\!\!=&\!\!\! \sum_{i=1}^\I\lambda_i^2\, v_i(x)^2\left(\frac{1-\e^{-\alpha (t-s)}}{2\alpha}\right)
+\sum_{i=1}^\I\lambda_i^2\, v_i(x)^2\left(\frac{(1-
\e^{-\alpha(t-s)})^2-(\e^{-\alpha t}-\e^{-\alpha s})^2}{2\alpha}\right).
\eeann
As the exponential function on the negative half axis is H\"older continuous for 
every $\rho\in[0,1]$ it holds
\beann
\E|O(x,t)-O(x,s)|^2&\leq& C_\rho\sum_{i=1}^\I\lambda_i^2\, v_i(x)^2 \,|t-s|^\rho\,.
\eeann
Thus, overall Jensen's inequality yields $\E|O(x,t)-O(y,s)|^2\leq C_\rho(|x-y|^2+|t-s|^2)^{\rho/2}$. 
Since the difference $O(x,t)-O(y,s)$ is centered Gaussian it further holds that
\benn
\E|O(x,t)-O(y,s)|^{2m}\,\leq\,C_{\rho,m}(|t-s|^2+|x-y|^2)^{m\rho/2}\qquad\forall\,m\in\N\,.
\eenn
Now, the Kolmogorov-Centsov Theorem implies the statement of the lemma.
\end{proof}

We present an example to illustrate the type of noise we are generally interested in. 
Further motivation is provided in Section \ref{sec:gain_perturb}. 

\begin{example2}
\label{ex:trigonometric_noise}
Consider the neural field equation on a $d$-dimensional cube $\dom=[0,2\pi]^d$
with noise based on trigonometric 
basis functions of $L^2([0,2\pi]^d)$. This type of noise is almost ubiquitous 
in applications as for the stochastic heat equations the basis functions can 
be chosen such that the usual (Dirichlet, Neumann or periodic) boundary conditions 
are preserved. For the example of noise preserving homogeneous Neumann boundary conditions 
the basis functions are
\be
v_i(x)=\prod_{k=1}^d e_{i_k}(x_k),
\ee
where $x=(x_1,\ldots,x_d)$, $i=(i_1,\ldots,i_k)$ is a multi-index in $\N^d$ and the functions $e_{i_k}$ are given by
\benn
e_{i_k}(x_k)=\left\{\begin{array}{cl} \displaystyle\frac{1}{\sqrt{2\pi}} & i_k=0,\\[2ex]
\displaystyle\frac{1}{\sqrt{\pi}}\,\cos(i_kx_k/2)& i_k\geq 1\,.
\end{array}\right.
\eenn 
The functions $v_i$ are for all $i\in\N^d$ pointwise bounded by $\pi^{-d/2}$ and Lipschitz 
continuous with Lipschitz constants given by $L_i=\pi^{-d/2}\,|i|$ (cf.~\cite[Lemma 5.3]{BloemkerJentzen}). 
Next, we construct a trace class Wiener process from these basis functions. A particular important 
example of spatio-temporal noise is smooth noise with exponentially decaying spatial correlation length 
\cite{BressloffWebber,KilpatrickErmentrout,GarciaSancho}, i.e.,
\be
\label{eq:exponentially_decaying_noise}
\E W_t(x)W_s(y) \, =\,\min\{t,s\}\,\frac{1}{(2\xi)^d}
\exp\Bigl(-\frac{\pi}{4}\frac{|x-y|^2}{\xi^2}\Bigr)\,+\,\textnormal{correction on the boundary}
\ee
for a parameter $\xi>0$ modelling the spatial correlation length. 
Note that for $\xi\to 0$ this noise process approximates space-time 
white noise. Following \cite{Shardlow} we can calculate under the assumption that $\xi\ll 2\pi$ 
the coefficients $\lambda_i^2$ such that the $Q$-Wiener process \eqref{eq:wiener_process} possesses 
the correlation function \eqref{eq:exponentially_decaying_noise} and obtain
\be
\lambda_i^2\,=\,\exp\Bigl(\frac{-\xi^2\,|i|^2}{4\pi}\Bigr)\,.
\ee
Now, it is easy to see that for this choice of eigenvalues the noise is of trace class and moreover the addtional conditions of Lemma \ref{lemma:spatial_continuity} are satisfied: As the functions $v_i$ are bounded we obtain
\beann
\sup_{x\in\dom}\left|\sum_{i\in \N^d}^\I \lambda_i^2\,v_i(x)^2\right| &\leq& \pi^{-d} + 
\pi^{-d} \sum_{N=1}^\I\ \sum_{i\in\{0,\ldots,N\}^d\backslash\{0,\ldots,N-1\}^d} \exp\Bigl(\frac{-\xi^2\,|i|^2}{4\pi}\Bigr)\\[1ex]
&\leq& \pi^{-d}+\pi^{-d} \sum_{N=0}^\I \exp\Bigl(\frac{-\xi^2\,N^2}{4\pi}\Bigr)\,2^{N-1}\\[1ex]
&<&\I
\eeann
and the second condition of \eqref{lemma:spatial_continuity_conditions} is satisfied as
\beann
\sup_{x\in\dom}\left|\sum_{i=\N^d} \lambda_i^2L_i^{2\rho}|v_i(x)|^{2(1-\rho)}\right| &\leq& \pi^{-d} 
\sum_{N=1}^\I\ \sum_{i\in\{0,\ldots,N\}^d\backslash\{0,\ldots,N-1\}^d} \exp\Bigl(\frac{-\xi^2\,|i|^2}{4\pi}\Bigr)\,|i|^{2\rho}\\[1ex]
&\leq& \pi^{-d}\sqrt{d} \sum_{N=0}^\I \exp\Bigl(\frac{-\xi^2\,N^2}{4\pi}\Bigr)\, N^{2\rho}2^{N-1}\\[1ex]
&<& \I\,.
\eeann
\end{example2}

\section{Gain Function Perturbation}
\label{sec:gain_perturb}

Another motivation for the considered additive noise neural field equations stems from 
a (formal) perturbation of the gain function $f$ with space-time white noise. 
Let $\dot \cW$ denote space time white noise and consider the randomly perturbed 
Amari equation
\be
\label{eq:white_noise_perturbation}
\partial_t U(x,t) = -\alpha U(x,t) +\int_\dom w(x,y)\Bigl(f(U(y,t))+
\epsilon\,\dot \cW(y,t)\Bigr)\,dy\,.
\ee
Recall that, by assumption (H2), the integral operator $K$ defined by the 
kernel $w$ is a self-adjoint compact operator. Thus the Spectral Theorem 
implies that $K$ possess only real eigenvalues $\lambda_i$, $i\in\N$, and 
the corresponding eigenfunctions $v_i$ form an orthonormal basis of $L^2(\dom)$. If 
additionally we assume that
\begin{enumerateII}
 \item[(H4)] $K$ is a Hilbert-Schmidt operator on $L^2(\dom)$, that is, 
$\|w\|_{L^2(\dom\times\dom)}<\infty$,
\end{enumerateII}
then the eigenvalues satisfy $\sum_{i=1}^\I \lambda_i^2<\infty$. Hence, $K$ 
possesses the series representation
\benn
Kg = \sum_{i=1}^\I \lambda_i\,\langle g,v_i \rangle \,v_i \qquad\forall\,g\in L^2(\dom)
\eenn
which yields for the perturbed equation \eqref{eq:white_noise_perturbation} 
the representation
\benn
\partial_t U(x,t) = -\alpha U(x,t) +\sum_{i=1}^\I \lambda_i\, 
\left(\int_\dom f(U(y,t)) v_i(y)\,dy +\epsilon\,\langle \dot \cW(t,\cdot),v_i \rangle \right)\, v_i(x)\,.
\eenn
Next, note that the random variables $\dot\beta^i_t=\langle\dot \cW(\cdot,t),v_i \rangle$ 
form a sequence of independent scalar white noise processes in time. Therefore, 
the perturbed equation becomes
\benn
\partial_t U(x,t) = -\alpha U(x,t) +\int_\dom w(x,y)\,f(U(y,t))\,dy + 
\epsilon\,\sum_{i=1}^\I \lambda_i\,\dot\beta^i_t\,v_i(x)\,.
\eenn
Re-writing this equation in the usual notation of stochastic differential equations we obtain
\be
d U_t(x) = \left[-\alpha U_t(x) +\int_\dom w(x,y)\,f(U_t(y))\,dy\right]dt + \epsilon\,d W_t(x)\,,
\ee
where 
\benn
W_t(x)=\sum_{i=1}^\I \lambda_i\,\beta^i_t\,v_i(x)
\eenn
is a trace-class Wiener process on the Hilbert space $L^2(\dom)$. Note, when 
comparing to \eqref{eq:wiener_process} here the coefficients $\lambda_i$ may be 
negative, however, as $-\beta^i$ is also a Wiener process this slight 
inconsistency can be neglected.\medskip

We next want to discuss spatial continuity of the solution to this equation 
with its particular noise structure. It is clear that this should translate 
into smoothing conditions of the kernel $w$. Due to Lemma 
\ref{lemma:spatial_continuity} it is sufficient to establish conditions 
\eqref{lemma:spatial_continuity_conditions}: First, it holds that
\benn
\sum_{i=1}^\I \lambda_i^2 v_i(x)^2=\sum_{i=1}^\I\left(\int_\dom w(x,y)v_i(y)dy\right)^2=
\sum_{i=1}^\I \langle w(x,\cdot),v_i\rangle^2
= \|w(x,\cdot)\|_{L^2(\dom)}^2
\eenn
due to Parseval's identity. Hence, the first condition of \eqref{lemma:spatial_continuity_conditions} 
becomes
\be
\label{eq:kernel_condition_1}
\sup_{x\in\dom} \|w(x,\cdot)\|_{L^2(\dom)}\,<\I.
\ee
Next, the basis functions are continuous if the kernel $w(x,y)$ is continuous 
in $x$ and as the minimal Lipschitz constant is given by the supremum on the 
derivatives we obtain
\benn
L_i\,=\,\sup_{x\in\dom}\left|\frac{1}{\lambda_i}\nabla_{\!x}
\int_\dom w(x,y)v_i(y)dy\right|\,\leq\,\frac{1}{|\lambda_i|} 
\sup_{x\in\dom}\|\nabla_{\!x} w(x,\cdot)\|_{L^2(\dom)}
\eenn
due to the Cauchy-Schwarz inequality. Therefore the second condition in 
\eqref{lemma:spatial_continuity_conditions} is satisfied if
\be
\label{eq:kernel_condition_2}
\sup_{x\in\dom} \|\nabla_{\!x} w(x,\cdot)\|_{L^2(\dom)}\,<\I\qquad\textnormal{and}
\qquad\sum_{i=1}^\I |\lambda_i|^{2(1-\rho)}|v_i(x)|^{2(1-\rho)}\leq M \quad\forall\,x\in\dom
\ee
for a $\rho\in(0,1)$ and a $M<\I$. The condition \eqref{eq:kernel_condition_1} and the first 
part of \eqref{eq:kernel_condition_2} are easily checked but for the second part 
of \eqref{eq:kernel_condition_2} usually theoretical results on the speed of 
decay of the eigenvalues have to be obtained. We note that \eqref{eq:kernel_condition_2} 
is certainly satisfied with $\rho=1/2$ if $K$ is a trace class operator and the eigenfunctions are pointwise bounded independently of $i$, see, e.g., Example \ref{ex:trigonometric_noise}.

\section{Deterministic Dynamics}
\label{sec:deterministic}

The classical deterministic Amari model, obtained for $\epsilon=0$ in \eqref{eq:Amari1}, is
\be
\label{eq:Amari1_det}
\partial_t U(x,t)=-\alpha U(x,t)+\int_{\dom} w(x,y)f(U(y,t))dy.
\ee
where $\dom\subseteq \R^d$. Note that we may allow $\dom$ to be unbounded for the deterministic case
as solutions of \eqref{eq:Amari1_det} do exist in this case \cite{PotthastBeimGraben}. Suppose 
there exists a stationary solution $U^*=U^*(x)$ of \eqref{eq:Amari1_det}. To determine the 
stability of $U^*$ consider $U(x,t)=U^*(x)+\psi(x,t)$. 
Substituting into \eqref{eq:Amari1_det} and Taylor-expanding around $U^*$ yields the 
linearized problem
\be
\label{eq:Amari1_det_lin}
\partial_t \psi(x,t)=-\alpha \psi(x,t)+\int_{\dom} w(x,y)(Df)(U^*(y))\psi(y,t)dy.
\ee 
Hence, the standard ansatz $\psi(x,t)=\psi_0(x)e^{\mu t}$ leads to the eigenvalue problem 
\be
\label{eq:Amari1_det_operator}
\underbrace{(\mu +\alpha)}_{=:\eta}\psi_0(x)=\int_{\dom} w(x,y)(Df)(U^*(y))\psi_0(y)dy:=
(\cL \psi_0)(x)\quad \text{or}\quad \cL \psi_0=\eta\psi_0.
\ee 
The linear stability condition $\mu<0$ is equivalent to $\eta<\alpha$ where $\eta\in \text{spec}(\cL)$. 
The stability analysis can be reduced to the understanding of the operator $\cL$. However, 
this is a highly non-trivial problem as the behaviour depends upon $\dom$, $U^*(x)$, $w(x,y)$ 
and $f(u)$.\medskip 

An LDP and Kramers' law are of particular interest in the case of bistability. Therefore, 
we point out that there are many situations where \eqref{eq:Amari1_det} does have three 
stationary solutions: $U^*_\pm(x)$ which are stable and $U^*_0(x)$ which is unstable. 
The following three examples make this claim more precise.

\begin{example2}
\label{ex:EM}
The first example is presented by Ermentrout and McLeod \cite{ErmentroutMcLeod}. Let $\dom=\R$, 
$w(x,y)=w(|x-y|)$, $\alpha=1$ and assume that $0\leq U(x,t)\leq 1$. Furthermore, suppose that 
$f\in C^1([0,1],\R)$ with $f'>0$ and 
\be
\label{eq:alg_EM}
\tilde{f}(U):=-U+f(U)
\ee
has precisely three zeros $U=0,a,1$ with $0<a<1$. The additional conditions $f'(0)<1$ and 
$f'(1)<1$ guarantee stability of the stationary solutions $U=0$ and $U=1$. As an even more 
explicit assumption \cite[p.463]{ErmentroutMcLeod} one may consider a
Dirac $\delta$-distribution for $w$ in \eqref{eq:Amari1_det} which yields 
\be
\label{eq:EM_ex}
\partial_t U(x,t)=-U(x,t)+F(U(x,t)).
\ee  
Suppose there are precisely three solutions for $U=F(U)$ given by $U=0,a,1$ with $0<a<1$. 
If $F'(0)<1$, $F'(1)<1$ and $F'(a)>1$ then \eqref{eq:EM_ex} has an unstable 
stationary solution between the two stable stationary solutions.  
\end{example2}

\begin{example2}
\label{ex:GC}
An even more concrete example is given by Guo and Chow \cite{GuoChow1,GuoChow2}. They assume 
$\dom=\R$, $w(x,y)=w(x-y)$, $\alpha=1$ and fix two functions
\benn
f(u)=[b(u-u_b)+1]H(u-u_b),\qquad w(x)=Ae^{-a|x|}-e^{|x|}
\eenn
where $H(\cdot)$ is the Heaviside function and $b$, $a$, $A$ and $u_b$ are parameters. Depending 
on parameter values one may obtain three constant stationary solutions exhibiting bistability as 
expected from Example \ref{ex:EM}. However, there are also parameter values so that three stationary 
pulses exhibiting bistability exist.
\end{example2}

Note that the choice $\dom=\R$ is not essential to obtain two deterministically-stable stationary states 
$U^*_\pm(x)$ and one deterministically-unstable stationary state $U^*_0(x)$. The important aspect is that
certain algebraic equations, such as $U=f(U)$ and $U=F(U)$ in Example \ref{ex:EM}, have the correct
number of solutions. Furthermore, one has to make sure that the sign of the nonlinearity $f$ is chosen 
correctly to obtain the desried deterministic stability results for the stationary solutions. Hence 
we expect that a similar situation also holds for bounded domains; see also \cite{VeltzFaugeras}.\medskip

Examples \ref{ex:EM}-\ref{ex:GC} are typical for many similar cases with $x\in\R$ 
or $x\in\R^2$. Many results on existence and stability of stationary solutions are 
available see {e.g.}~\cite{Amari,LaingTroy,KubotaHamaguchiAihara,LaingTroyGutkinErmentrout} 
and references therein.

\begin{example2}
As a higher-dimensional example one may consider the work by Jin, Liang and Peng 
\cite{JinLiangPeng} who assume that $w(x,y)=w(x-y)$, $\alpha=1$, $\dom=\R^d$ and 
\benn
Z_\I=\int_{\R^d} w(x)~dx<\I, \qquad \kappa Z_\I>1,
\eenn 
where $\kappa$ is the Lipschitz constant of $f\in C^1(\R^d,\R)$. Furthermore, suppose 
$f'$ is uniformly continuous and 
\benn
\begin{array}{lcll}
f'(U)Z_\I&<&1 \quad & \text{for $U\in (-\I,U_1)\cup (U_2,\I)$},\\
f'(U)Z_\I&=&1 \quad & \text{for $U\in \{U_1,U_2\}$},\\
f'(U)Z_\I&>&1 \quad & \text{for $U\in (U_1,U_2)$},\\
\end{array}
\eenn
for $U_1<0< U_2$. Then \cite[Prop.11]{JinLiangPeng} the conditions
\benn
-U_1+f(U_1)Z_\I<0\qquad \text{and}\qquad -U_2+f(U_2)Z_\I>0 
\eenn 
yield three stationary solutions $U^*_+$, $U^*_-$ and $U^*_0$. The solutions 
$U^*_\pm$ are stable and satisfy $U^*_-\leq 0$ and $U^*_+>0$. The solution $U^*_0$ 
is unstable.
\end{example2}

Although we only focus on stationary solutions, it is important to remark that the 
techniques developed here could - in principle - also be applied to travelling waves 
$U(x,t)=U(x-st)$ for $s>0$. The existence and stability of travelling waves for 
\eqref{eq:Amari1_det} has been investigated for many different situations; see 
{e.g.}~\cite{BressloffReview,ErmentroutMcLeod,Bressloff1,CoombesOwen} and references 
therein. However, it seems reasonable to restrict ourselves here to the stationary 
case as even for this simpler case an LDP and Kramers' law are not yet well understood.

\section{Large Deviations and Kramers' Law}
\label{sec:LDP}

Here we briefly introduce the background and notation for LDPs and Kramers' law needed 
through the remaining part of the paper; see \cite{FreidlinWentzell,DemboZeitouni} for 
more details. Consider a topological space $\cX$ with Borel $\sigma$-algebra $\cB_\cX$. 
A mapping $I:\cX\ra [0,\I]$ is called a good rate function if it is lower semicontinuous 
and the level set $\{h:I(h)\leq \alpha\}$ is compact for each $\alpha\in[0,\I)$. Sometimes 
the term action functional is used instead of rate function. Consider a family 
$\{\mu^\epsilon\}$ of probability measures on $(\cX,\cB_\cX)$. The measures $\{\mu^\epsilon\}$ 
satisfy an LDP with good rate function $I$ if 
\be
\label{eq:LDP}
-\inf_{\Gamma^o}I\leq \liminf_{\epsilon\ra 0}\epsilon^2\ln \mu^\epsilon(\Gamma)\leq
\limsup_{\epsilon\ra 0}\epsilon^2\ln \mu^\epsilon(\Gamma)\leq -\inf_{\bar{\Gamma}}I 
\ee
holds for any measurable set $\Gamma\subset \cX$; often infima over the interior 
$\Gamma^o$ and closure $\bar{\Gamma}$ coincide so that $\liminf$ and $\limsup$ coincide 
at a common limit. One of the most classical cases is the application of \eqref{eq:LDP} 
to finite-dimensional SDEs
\be
\label{eq:gen_SDE_ld}
du_t=g(u_t)dt+\epsilon G(u_t) d\beta_t
\ee
where $u_t\in\R^N$, $g:\R^N\ra \R^N$, $G:\R^N\ra \R^{N\times k}$,  
$\beta_t=(\beta_t^{1},\ldots,\beta_t^{k})^T$ is a vector of $k$ independent Brownian 
motions and we shall assume that the initial condition $u_0\in\R^N$ is deterministic. If we want to 
emphasize that $u_t$ depends on $\epsilon$ we shall also use the notation $u^\epsilon_t$. 
The topological 
space is chosen as a path space
\benn
\cX:=C_0([0,T],\R^N)=\{\phi\in C([0,T],\R^N):\phi(0)=u_0\}.
\eenn
To state the next result we also need the Sobolev space
\be
\label{eq:Sobolev_intro}
H_1^N:=\{\phi:[0,T]\ra \R^N:\text{$\phi$ absolutely continuous, $\phi'\in L^2$, $\phi(0)=0$}\}.
\ee
Furthermore, we are going to assume that the diffusion matrix 
$\mathfrak{D}(u):=G(u)^TG(u)\in\R^{N\times N}$ is positive definite.

\begin{theorem}(\cite{FreidlinWentzell},\cite{DemboZeitouni})
\label{thm:SDE_LDP}
The SDE \eqref{eq:gen_SDE_ld} satisfies the LDP \eqref{eq:LDP} given by
\be
\label{eq:LDP1}
-\inf_{\Gamma^o}I\leq \liminf_{\epsilon\ra 0}\epsilon^2\ln \P((u_t^\epsilon)_{t\in[0,T]}\in\Gamma)\leq
\limsup_{\epsilon\ra 0}\epsilon^2\ln \P((u_t^\epsilon)_{t\in[0,T]}\in\Gamma)\leq -\inf_{\bar{\Gamma}}I. 
\ee
for any measurable set of paths $\Gamma\subset \cX$ with good rate function
\be
\label{eq:LDP_rate_fd}
I(\phi)=I_{[0,T]}(\phi)=\left\{
\begin{array}{ll}
 \frac12 \int_0^T(\phi'_t-g(\phi_t))^T\mathfrak{D}(\phi_t)^{-1}(\phi_t'-g(\phi_t))dt,
 & \text{if $\phi\in u_0+H_1^N$,}\\[1ex]
+\I  & \text{otherwise.}
\end{array}
\right.
\ee
\end{theorem}

An important application of the LDP \eqref{eq:LDP1} is the so-called 
first-exit problem. Suppose that $u_t$ starts near a stable equilibrium 
$u^*\in \cD\subset \R^N$ of the deterministic system given by setting $\epsilon=0$
in \eqref{eq:gen_SDE_ld}, where $\cD$ is a bounded domain with smooth boundary. 
Define the first-exit time 
\be
\label{eq:ld_basic_fe}
\tau^\epsilon_\cD:=\inf \{t>0:u^\epsilon_t\not\in \cD\}.
\ee 
To formalize the application of the LDP define the mapping
\be
Z(u,v;s):=\inf\{I(\phi):\phi\in C([0,s],\R^N),\phi_0=u,\phi_s=v\}
\ee  
which is the cost for a path starting at $u$ to reach $v$ in time $s$. Next assume that $\bar{\cD}$ is properly contained inside the (deterministic) basin of attraction of $u^*$. Then one can show \cite[Thm.~4.1, p.124]{FreidlinWentzell} that
\be
\lim_{\epsilon\ra 0} \epsilon^2 \ln \P(\tau^\epsilon_{\cD}\leq t|u=u_0)=
\inf\{Z(u,v;s):s\in [0,t],v\not\in \cD\}.
\ee
To get a more precise information on the exit distribution one defines the function
\benn
Z(u^*,v)=\inf_{t> 0} Z(u^*,v;t)
\eenn 
which is called the quasipotential for $u^*$. It is natural to minimize the quasipotential over 
$\partial\cD$ and define
\benn
\bar{Z}:=\inf_{v\in\partial \cD}Z(u^*,v).
\eenn

\begin{theorem}(\cite[Thm.~4.2, p.127]{FreidlinWentzell},\cite[Thm.~5.7.11]{DemboZeitouni})
\label{thm:LDP_exit}
For all initial conditions $u\in \cD$ and all $\delta>0$, the following two limits hold:
\bea
&&\lim_{\epsilon\ra 0}\P\left(e^{(\bar{Z}-\delta)/\epsilon^2}<\tau^\epsilon_\cD<
e^{(\bar{Z}+\delta)/\epsilon^2}|u_0=u^*\right)=1,\label{eq:LDP_exit1}\\
&&\lim_{\epsilon\ra 0} \epsilon^2 \ln \E[\tau^\epsilon_\cD|u_0=u^*]=\bar{Z}.  \label{eq:LDP_exit2}
\eea
\end{theorem}

If the SDE \eqref{eq:gen_SDE_ld} has a gradient structure with identity diffusion matrix {i.e.}
\be
g(u)=-\nabla V(u) ~ \text{ for $V:\R^N\ra \R$},\qquad \text{and}\qquad G(u)=Id\in\R^{N\times N}
\ee
then one can show \cite[Sec.~4.3]{FreidlinWentzell} that the quasipotential is given by 
$Z(u^*,v)=2(V(v)-V(u^*))$. If the potential has precisely two local minima $u^*_\pm$ and 
a saddle point $u^*_s$ with $N-1$ stable directions so that the Hessian
$\nabla^2V(u^*_s)$ has eigenvalues
\benn
\rho_1(u^*_s)<0<\rho_2(u^*_s)<\cdots <\rho_N(u^*_s)
\eenn
then one can even refine Theorem \ref{thm:LDP_exit}. Suppose $u_0=u^*_-$ then the mean 
first passage time to $u^*_+$ 
satisfies
\be
\label{eq:Kramers}
\E[\inf\{t>0:\|u_t-u^*_+\|_2\leq \delta\}]\sim
 \frac{2\pi}{|\rho_1(u^*_s)|} \sqrt{\frac{|\det(\nabla^2V(u_s^*))|}{\det(\nabla^2V(u_-^*))}}
 \e^{2(V(u^*_s)-V(u^*_-))/\epsilon^2}
\ee
where $\|\cdot\|_2$ denotes the usual Euclidean norm in $\R^N$. The formula \eqref{eq:Kramers} is also known as 
Kramers' law \cite{Berglund3} or Arrhenius-Eyring-Kramers' 
law \cite{Arrhenius,Eyring,Kramers}. Note that the key differences with the general
LDP \eqref{eq:LDP_exit1} for the first-exit problem are that \eqref{eq:Kramers}
yields a precise prefactor for the exponential transition time and uses the explicit 
form of the good rate function for gradient systems. It is interesting to note that a 
rigorous proof of \eqref{eq:Kramers} has only been obtained quite recently 
\cite{BovierEckhoffGayrardKlein,BovierGayrardKlein}.

\section{Gradient Structures in Infinite Dimensions}
\label{sec:gradient}

The finite-dimensional Kramers' formula \eqref{eq:Kramers} applies to SDEs 
\eqref{eq:gen_SDE_ld} with a gradient-structure $g(u)=-\nabla V(u)$ where 
$V:\R^N\ra \R$ is the potential. A generalization of Kramers' law has been 
carried over to the infinite-dimensional case of SPDEs given by
\be
\label{eq:SPDE_BG}
d U=[\Delta U-h'(U)]dt+\epsilon~ dW(x,t)
\ee
for $U=U(x,t)$, $x\in\tilde{\dom}\subset \R$, $\tilde{\dom}$ a bounded interval, $h\in C^k(\R,\R)$ 
for suitably large $k\in\N$ and $W(x,t)$ denotes space-time white noise and 
either Dirichlet or Neumann boundary conditions are used 
\cite{Barret,BerglundGentz11,BerglundGentz10}. A crucial reason why this 
generalization works is that the SPDE \eqref{eq:SPDE_BG} has a gradient-type 
structure \cite{FarisJona-Lasinio} given by the energy functional
\be
\label{eq:energy_BG}
V[U]:=\int_{\tilde{\dom}}\left[\frac12 U'(x)^2+h(U(x))\right]dx.
\ee
More precisely, when $\epsilon=0$ one obtains from \eqref{eq:SPDE_BG} a PDE, 
say with Dirichlet boundary conditions,
\be
\label{eq:SPDE_BG_PDE}
d U=[\Delta U-h'(U)]dt,\qquad U(x)=0\text{ on $\partial \tilde{\dom}$}
\ee
for a given sufficiently smooth initial condition $U(x,0)=U_0(x)\in C^k(\R,\R)$. 
Standard parabolic regularity \cite[Sec.~7.1]{Evans} implies that solutions $U$ 
of \eqref{eq:SPDE_BG_PDE} lie in the Sobolev spaces $H^k_0(\tilde{\dom})$. Computing 
the G\^{a}teaux derivative in this space yields
\be
\label{eq:Berglund_D}
\nabla_zV[U]=\int_{\tilde{\dom}}[-U''(x)+h'(U(x))]z(x)~dx.
\ee
The G\^{a}teaux derivative is equal to the Fr\'{e}chet derivative $\nabla V=DV$ 
by a standard continuity result \cite[p.47]{Deimling}. Hence \eqref{eq:Berglund_D} 
shows that the stationary solutions of \eqref{eq:SPDE_BG_PDE} are critical points 
of the gradient functional $V$. Since the gradient structure of the deterministic 
PDE \eqref{eq:SPDE_BG_PDE} is a key structure to obtain a Kramers'-type estimate for
the SPDE \eqref{eq:SPDE_BG} we would like to check whether there is an analogue available
for the deterministic Amari model \eqref{eq:Amari1_det}.\medskip

We shall assume for simplicity that $f\in BC^1(\R)$ for the calculations in this section.
Although this is a slightly stronger assumption than (H1) we shall see below that even
with this assumption we are not able to obtain an immediate generalization of \eqref{eq:Berglund_D}. Using a direct 
modification of the results in \cite{PotthastBeimGraben} it follows that the deterministic
Amari model \eqref{eq:Amari1_det} has solutions $U(x,t)$ in the H{\"{o}}lder space 
$BC^\alpha(\dom)\times BC^\alpha([0,T])$ for $\alpha\in (0,1]$ and $\dom\subset \R^d$ is the 
usual domain we use for the Amari model. Now consider the analogous naive 
guess to \eqref{eq:Berglund_D} given by
\be
\label{eq:energy_KR}
V[U]:=\int_{\dom} \left[\frac{\alpha}{2}U(x)^2-\int_{\dom}\int_0^{U(y)}f(r)w(x,y)drdy\right]dx.
\ee
Computing the derivative in $BC^\alpha(\dom)$ yields
\bea
\nabla_zV[U]&=&\lim_{\delta\ra 0}\frac1\delta \left(V[U+\delta z]-V[U]\right),\nonumber\\
&=& \int_{\dom} \left[\alpha U(x)z(x)-\int_{\dom}f(U(y))w(x,y)z(y)dy\right]dx.\label{eq:mixed_var}
\eea
Therefore, setting $\nabla_zV[U]=0$ is not equivalent to the solution of the stationary problem
\benn
-\alpha U(x)+\int_{\dom} w(x,y)f(U(y))dy=0.
\eenn
Due to the presence of the different terms $z(x)$ and $z(y)$ in \eqref{eq:mixed_var} one may guess
that the modified functional
\be
\label{eq:energy_KR1}
V[U]:=\int_{\dom} \left[\frac{\alpha}{2}U(x)^2-\frac12\int_{\dom}f(U(y))f(U(x))w(x,y)dy\right]dx
\ee
could work. However, another direct computation shows that 
\beann
\nabla_zV[U]&=&\int_{\dom}\left[\alpha U(x)z(x)dx\right]-
\frac12\left[\int_{\dom}\int_{\dom}f(U(x))Df(U(y))w(x,y)z(y)dy ~dx\right]\\
&& -\frac12\left[\int_{\dom}\int_{\dom}f(U(y))Df(U(x))w(x,y)z(x)dy ~dx\right]\\
&=&\alpha\langle U,z\rangle-\langle KF(U),Df(U)z \rangle.
\eeann
Hence $f$ and its derivative $Df$ both appear instead of the desired formulation; 
by a similar computation one can show that replacing $f(u(\cdot))$ in 
\eqref{eq:energy_KR1} by $\int_0^uf(r)dr$ fails as well. Hence there does not 
seem to be a natural generalization for the guess for the gradient functional \eqref{eq:energy_BG}. 
However, one has to consider possible
coordinate changes. The idea to apply a preliminary transformation has been discussed 
{e.g.}~in \cite[p.2]{EnculescuBestehorn} and \cite[p.488]{LaingTroy}. Assume that
\be
\label{eq:assume_EB}
f^{-1}=:g~\text{exists and}~g'\neq 0.
\ee
Define $P(x,t):=f(U(x,t))$ as the mean action-potential generating rate so that $U=g(P)$. Observe that
\be
\partial_t P(y,t)=\frac{1}{g'(P(x,t))}\left[-\alpha g(P(x,t))+\int_\dom w(x,y)P(y,t)dy\right].
\ee
For this equation the problem observed in \eqref{eq:energy_KR1} should disappear as the integral only contains 
linear terms. One may define an energy-type functional 
\benn
E[P]:=\int_\dom\left[\int_0^{P(x)}\alpha g(r)dr-\frac12\int_\dom w(x,y)P(y)P(x)dy\right]dx.
\eenn
Calculating the derivative yields
\beann
\nabla_QE[P]&=&\lim_{\delta\ra 0}\frac1\delta \left(E[P+\delta Q]-E[P]\right),\\
&=& \int_\dom \alpha g(P(x))Q(x)dx-\frac12\int_\dom\int_\dom w(x,y)\{P(y)Q(x)+P(x)Q(y)\}dy ~dx,\\
&=&\langle\alpha g(P),Q\rangle-\left\langle \int_\dom w(x,y) P(y)dy,Q\right\rangle.
\eeann
This shows that there is hidden energy-type flow structure in the Amari model for the assumptions 
\eqref{eq:assume_EB} so that
\be
\label{eq:Bestehorn}
\partial_t P(x,t)=-\frac{1}{g'(P(x,t))}\nabla E[P(x,t)].
\ee
However, even with this variable transformation there seems to be little hope to derive a precise Kramers'
rule for the stochastic Amari model \eqref{eq:Amari1} by generalizing the approach for SPDE systems 
\cite{Barret,BerglundGentz11,BerglundGentz10}. The problems are as follows:

\begin{itemize}
 \item There is still a space-time dependent nonlinear prefactor $1/g'(P(x,t))$ in 
 \eqref{eq:Bestehorn} for the deterministic system, so the system is not an exact 
 gradient flow for a potential. 
 \item Applying the change-of-variable $P_t(x):=f(U_t(x))$ 
for the stochastic Amari model \eqref{eq:Amari1} requires an It\^{o}-type formula so that 
\be
\label{eq:Amari1_transformed}
dP_t(x)=\frac{1}{g'(P_t(x))}\left[-\alpha g(P_t(x))+\int_{\dom} w(x,y)P_t(y)dy+
\cO(\epsilon^2)\right]dt+\epsilon M(P_t(x))dW_t(x),
\ee
where $M(P_t(x))$ is now a multiplicative noise term; see \cite{DaPratoJentzenRoeckner} and 
references therein for more details on infinite-dimensional It\^{o}-type formulas. The higher-order 
term $\cO(\epsilon^2)$ in the drift part of \eqref{eq:Amari1_transformed} is not expected to 
cause difficulties but a multiplicative noise structure definitely excludes the direct application 
of Kramers' law. 
 \item Even if we would just assume - without any immediate physical motivation - that the noise term in 
\eqref{eq:Amari1_transformed} is purely additive $\epsilon dW_t(x)$ there is a problem to apply 
Kramers' law since we do not have a structure like in \eqref{eq:gen_SDE_ld} with $G(\cdot)=Id$ as
$W_t(x)$ is a $Q$-Wiener process defined in \eqref{eq:wiener_process} and driving space-time white 
noise in \eqref{eq:Amari2} is particularly excluded due to the non-existence of a solution. 
\end{itemize}

Based on these observations an immediate approach to generalize a sharp Kramers' formula to 
neural fields seems unlikely. Hence we try to understand an LDP for the stochastic Amari-type 
model \eqref{eq:Amari1}. 

\section{Direct Approach to an LDP}
\label{sec:direct}

A general direct approach for the derivation of an LDP for infinite-dimensional stochastic 
evolution equations is presented in \cite{daPratoZabczyk} and further results have been obtained for certain additional classes of SPDEs 
\cite{Cardon-Weber,CerraiRoeckner2,CerraiRoeckner3,RoecknerWangWu}.
The results in \cite{daPratoZabczyk} are valid for semilinear equations with suitable 
Lipschitz assumptions on the nonlinearity and with solutions taking values in $C(\cD)$. 
We state the available results applied to continuous solutions 
of the Amari equation \eqref{eq:Amari2} assuming that the conditions of Lemma 
\ref{lemma:spatial_continuity} are satisfied.\medskip

For the following we assume that there exists an open neighbourhood $\cD\in C(\dom)$ containing a 
stable equilibrium state $u^\ast$ of the deterministic Amari equation \eqref{eq:Amari1_det} 
such that $\bar\cD$ is contained in the basin of attraction of $u^\ast$. We are interested in the
rate function and the first-exit time of the process from $\cD$ given by 
\benn
\tau^\epsilon_\cD=\inf\{t\geq 0: U\notin \cD\}
\eenn 
if $U$ starts in the deterministic equilibrium state $u^\ast$. In order to state the 
quasipotential for $u$ we consider the control system
\be
\label{eq:control_system}
\dot y = -\alpha y+KF(y)+Q^{1/2}v, \quad y_0=x\in C(\dom)
\ee
for controls $v\in L^2((0,T),L^2(\dom))$ for all $T>0$ and denote by $y^{x,v}$ its 
unique mild solution\footnote{The existence of such a solution is guaranteed by 
standard results on deterministic equations (cf.\cite[Sect.~A.3]{daPratoZabczyk}) as 
long as $Q^{1/2}$ maps $L^2(\dom)$ continuously into $C(\dom)$. This is easily established. 
The unique square root $Q^{1/2}$ of $Q$ is the Hilbert-Schmidt operator given by 
$Q^{1/2} g=\sum_{i=1}^\I \lambda_i \langle v_i,g \rangle v_i$ for all $g\in L^2(\dom)$ and in order 
to show that $Q^{1/2}g\in C(\dom)$ it remains to establish that the functions converge 
uniformly on $\dom$. This holds as for all $x\in\dom$
\benn
\Big|\sum_{i=N}^\I \lambda_i \langle v_i,g\rangle v_i(x)\Big|\,\leq\, 
\Big|\sum_{i=N}^\I \lambda_i^2v_i(x)^2\Big|^{1/2}\Big|\sum_{i=1}^\I \langle v_i,g\rangle^2\Big|^{1/2}\,
\leq\,\Big(\sup_{x\in\dom}\Big|\sum_{i=1}^\I \lambda_i^2v_i(x)^2\Big|\Big)^{1/2}
\Big(\sum_{i=N}^\I \langle v_i,g\rangle^2\Big)^{1/2}.
\eenn
Hence the upper bound, which is finite due to \eqref{lemma:spatial_continuity_conditions}, 
is in independent of $x$ and converges to zero for $N\to\infty$. Moreover, we further 
find that $g(t)\in L^2((0,T),L^2(\dom))$ implies $Q^{1/2}g(t)\in L^2((0,T),C(\dom))$.} 
taking values in $C([0,T],C(\dom))$ for all $T>0$. Then we define
\be
I(u^\ast,z)=\inf\left\{\frac{1}{2}\int_0^T \|v(s)\|_{L^2}^2\, ds:\, y^{u^\ast,v}(T)=z,\,T>0\right\}\,, 
\ee
where this quasipotential relates to the minimal energy necessary to move the 
control system \eqref{eq:control_system} started at the equilibrium state $u^\ast$ to $z$.

\begin{theorem}(\cite[Thm.~12.18]{daPratoZabczyk}) 
\label{thm:daPrato_LDP}
It holds that
\benn
\lim_{\epsilon\to 0} \epsilon^2 \ln \E\bigl[\tau^\epsilon_\cD\big|U_0=u^\ast\bigr]\ =\  
\inf_{z\in\partial\cD} I(u^\ast,z)\,.
\eenn
\end{theorem}

Following further the exposition in \cite[Sec.~12]{daPratoZabczyk} explicit formulae for the rate function $I$ are 
only available in the special case of the drift possessing gradient structure and space-time white noise. As we 
have argued above, this structure is particularly not satisfied for neural field equations. Hence, the same 
observations as presented at the end of the last section prevent a further direct analytic approach to the LDP. 
Therefore, we try to understand the LDP problem for a discretized approximate finite-dimensional version of the 
neural field equation.

\section{Galerkin Approximation}
\label{sec:Galerkin}

Throughout the section we assume that the assumptions (H1)--(H3) are satisfied. As a 
discretized version of the neural field equation \eqref{eq:Amari1} we consider its spectral Galerkin 
approximations; recall that the solution $U_t$ of \eqref{eq:Amari1} lies in $C([0,T],L^2(\dom))$ as 
discussed in Section \eqref{sec:Amari}. In order to decouple the noise, we define the spectral 
representation of the solution
\be
\label{eq:spectral_representation}
U_t(x)=\sum_{i=1}^\I u_t^i ~v_i(x)\,.
\ee
Here the orthonormal basis functions $v_i$ are given by the eigenfunctions of the 
covariance operator of the noise with corresponding eigenvalues $\lambda_i^2$, see 
equation \eqref{eq:wiener_process}. To obtain a equation for the coefficients $u^i_t$ 
we take the inner product of equation \eqref{eq:Amari2} with the basis functions $v_i$ 
which yields
\benn
\langle dU_t,v_i\rangle =\bigl[-\alpha \langle U_t,v_i\rangle +\langle KF(U_t),v_i \rangle\bigr]dt+\epsilon \langle dW_t,v_i\rangle  \qquad\text{for $i\in\N$}.
\eenn
After plugging in \eqref{eq:spectral_representation} we obtain for $u^i$ the countable 
Galerkin system
\be
\label{eq:galerkin_system}
du^i_t=\bigl[-\alpha u_t^i+(KF)^i(u_t^1,u_t^2,\ldots)\bigr]dt+
\epsilon \lambda_i d\beta_t^i\qquad\text{for $i\in\N$}.
\ee
Here the non-linearities coupling all the equations are given by
\beann
(KF)^i(u_t^1,u_t^2,\ldots)&:=&\int_\dom v_i(x) 
\left(\int_\dom w(x,y) ~f\left(\sum_{j=1}^\I u_t^j ~v_j(y)\right) ~dy\right) ~dx\nonumber\\[1ex]
&=& \int_{\dom} f\left(\sum_{j=1}^\I u_t^j ~v_j(x)\right)\left(\int_{\dom} w(x,y)v_i(y)dy\right)dx
\eeann
due to the symmetry of the kernel $w$. If, in addition, we assume that (H4) holds and $K$ and $Q$ 
possess the same eigenfunctions and the eigenvalues are related as discussed in Section 
\ref{sec:gain_perturb} the non-linearities become
\be
(KF)^i(u_t^1,u_t^2,\ldots)=\lambda_i\int_{\dom} f\left(\sum_{j=1}^\I u_t^j ~v_j(x)\right)v_i(x)dx\,.
\ee
The $N$-th Galerkin approximation $U^N$ to $U$ is obtained truncating the spectral 
representation \eqref{eq:spectral_representation} and thus given by
\be
U^N_t=\sum_{i=1}^N u^{i,N}_t ~v_i\,,
\ee
where $u^{i,N}_t$ are the solutions to the $N$-dimensional Galerkin SDE system
\be
\label{eq:galerkin_approx_system}
du^{i,N}_t=\bigl[-\alpha u_t^{i,N}+(KF)^{i,N}(u_t^{1,N},\ldots,u_t^{N,N})\bigr]dt+
\epsilon \lambda_i d\beta_t^i\qquad\forall\,i=1,\ldots,N,
\ee
where the non-linearities $KF^{i,N}$ are given by 
\be
(KF)^{i,N}(u^{1,N},\ldots,u^{N,N})\,=\, \int_{\dom} f\left(\sum_{j=1}^N u^{j,N} 
~v_j(x)\right)\left(\int_{\dom} w(x,y)v_i(y)dy\right)dx
\ee
or, in the special case of Section \ref{sec:gain_perturb}, by
\be
(KF)^{i,N}(u^{1,N},\ldots,u^{N,N})=\lambda_i\int_{\dom} f\left(\sum_{j=1}^N u^{j,N} 
~v_j(x)\right)v_i(x)dx\,,
\ee
respectively. The following theorem establishes the almost sure convergence of the 
Galerkin approximations to the solution of \eqref{eq:Amari2}. Therefore, we may be 
able to infer properties of the behaviour of paths of the solution from the path 
behaviour of the Galerkin approximations. We have deferred the proof of the theorem 
to the appendix.

\begin{theorem} 
\label{thm:galerkin_convergence}
It holds for all $T>0$ that
\benn
\lim_{N\to\infty} \sup\nolimits_{t\in[0,T]}\|U_t-U^N_t\|_{L^2(\dom)}=0\quad a.s.
\eenn
If, in addition, the series $\sum_{i=1}^\I \lambda_i^2v_i^2$ converges in 
$C(\dom)$ and the functions $v_i$ are Lipschitz continuous with Lipschitz 
constants $L_i$ such that $\sup_x{\in\dom}\sum_{i=1}^\I 
\lambda_i^2 L_i^{2\rho}|v_i(x)|^{2(1-\rho)}<\I$ for a $\rho\in(0,1)$ (i.e.~the conditions of Lemma \ref{lemma:spatial_continuity} 
are satisfied), $U_0\in C(\dom)$ such that $\lim_{N\to\infty}\|U_0-P^NU_0\|_0=0$ and 
$K$ is compact on $C(\dom)$, then it holds for all $T>0$ that
\benn 
\lim_{N\to\infty} \sup\nolimits_{t\in[0,T]}\|U_t-U^N_t\|_0=0\quad a.s.
\eenn
\end{theorem}

\section{Approximating the LDP}
\label{sec:approximation}

The LDP in Theorem \ref{thm:daPrato_LDP} is not immediately computable. Here we show that a
finite-dimensional approximation can be made and what the structure of this approximation
entails. For simplicity, consider the case when the diagonal diffusion matrix $\mathfrak{D}$
with entries $\mathfrak{D}_{ii}=\lambda_i^2$ is positive definite {i.e.}~$\lambda_i\neq 0$
for all $i\in\N$. Observe that the inverse of $\mathfrak{D}$ induces an inner product on
$\R^N$ for $N\in\N\cup\{\I\}$ via
\benn
\langle a,b \rangle_{N}:=a^T~(\mathfrak{D})^{-1}~b=
[\mathfrak{D}^{-1/2}a]^T~[\mathfrak{D}^{-1/2}b]\qquad \text{for $a,b\in\R^N$},
\eenn
where $\mathfrak{D}$ is understood as the projection onto $\R^{N\times N}$ if $N<\I$. We are also going 
to use the notation introduced in Section \ref{sec:Galerkin} for the Galerkin approximation 
{i.e.}~$u^{\cdot,N}_t$ denotes the vector 
\be
\label{eq:ut_Galerkin}
(u^{1,N}_t,u^{2,N}_t,\ldots,u^{N,N}_t)^T\in\R^N
\ee
where $u^{\cdot,N}_t$ denotes the solutions of the $N$-dimensional system \eqref{eq:galerkin_approx_system}.
Note that throughout this section we shall always work with the Galerkin coefficients, {e.g.}, $u_t$ refers to the vector 
\benn
(u^1_t,u^2_t,\ldots)^T\in\R^\I.
\eenn
Furthermore, for arbitrary functions $\phi_t\in L^2(\dom)$, which are used in the formulation 
of the rate function, we use the notation $\phi^{\cdot,N}_t$ to denote the projection onto the first 
$N$ Galerkin coefficients. Theorem \ref{thm:SDE_LDP} immediately implies the following:

\begin{proposition}
\label{prop:our_fd_LDP}
For the finite-dimensional Galerkin system \eqref{eq:galerkin_approx_system} the rate function is given by
\be
\label{eq:our_LDP_rate_fd}
I^N(\phi^{\cdot,N})=\left\{
\begin{array}{ll}
 \frac12\int_0^T \langle (\phi^{\cdot,N}_t)'-g^{\cdot,N}(\phi^{\cdot,N}_t),
 (\phi^{\cdot,N}_t)'-g^{\cdot,N}(\phi^{\cdot,N}_t)\rangle_N~ dt,& \text{if $\phi\in u^{\cdot,N}_0+H_1^N$,}\\[1ex]
+\I  & \text{otherwise,}
\end{array}
\right.
\ee
where $g^{i,N}(\phi^{\cdot,N}_t)=-\alpha \phi_t^{i,N}+(KF)^{i,N}(\phi_t^{1,N},\ldots,\phi_t^{N,N})$.
\end{proposition}

Recall from Section \ref{sec:direct} that Theorem \ref{thm:daPrato_LDP} provides a large deviation principle. 
For the case when $Q$ is a positive operator, we may formally re-write the control system \eqref{eq:control_system} 
as 
\be
\mathfrak{D}^{-1/2}[\dot y -(-\alpha y+KF(y))]= v
\ee 
so that the rate function for the Amari model can be expressed as 
\be
\label{eq:our_LDP_rate_ifd}
I(\phi)=\left\{
\begin{array}{ll}
 \frac12\int_0^T \int_\dom \mathfrak{D}^{-1/2}[\phi_t'-g(\phi_t)]~
 \mathfrak{D}^{-1/2}[\phi_t'-g(\phi_t)]~ dx~ dt,& \text{if $\phi\in u_0+H_1^\I$,}\\[1ex]
+\I  & \text{otherwise,}
\end{array}
\right.
\ee
where $g(\phi_t)=-\alpha \phi_t+KF(\phi_t)$ and $\mathfrak{D}^{-1/2}u=\sum_{i=1}^\I(\mathfrak{D}^{-1/2}u,v_i)v_i$.
Therefore, the next result just implies that the Galerkin approximation is consistent for the LDP.

\begin{proposition}
\label{prop:Gal_LDP_conv}
For each $\phi_t\in u_0+H_1^\I$ we have $\lim_{N\ra \I} |I(\phi_t)-I^N(\phi^{\cdot,N}_t)|=0$.
\end{proposition}

\begin{proof}
Considering the finite-dimensional rate function \eqref{eq:our_LDP_rate_fd} it suffices to notice that
\beann
\langle (\phi^{\cdot,N}_t)'-g^{\cdot,N}(\phi^{\cdot,N}_t),(\phi^{\cdot,N}_t)'-g^{\cdot,N}(\phi^{\cdot,N}_t)\rangle_N
&=&\sum_{i=1}^N \frac{1}{\lambda_i^2}\left[(\phi^{i,N}_t)'-g^{i,N}(\phi^{\cdot,N}_t)\right]^2\\
&=&\sum_{i=1}^N \int_\dom \frac{1}{\lambda_i^2}\left[(\phi^{i,N}_tv_i(x))'-g^{i,N}(\phi^{\cdot,N}_t)v_i(x)\right]^2~dx
\eeann
by orthonormality of the basis in $L^2(\dom)$.
\end{proof}

Hence we may work with the finite-dimensional Galerkin system and its LDP for computational purposes. However,
the truncation $N$ may still be very large. We are going to show, using a formal analysis for a certain case, 
that there is an intrinsic multi-scale structure of the rate function. We assume that we are in the special case
considered in Section \ref{sec:gain_perturb} where $K$ and $Q$ have the same eigenfunctions and the
corresponding eigenvalues are given by $\lambda_i$ and $\lambda_i^2$, respectively. 

\begin{lemma}
\label{lem:3terms}
For each $N\in\N$ the first part of the rate function \eqref{eq:our_LDP_rate_fd} can be re-written as
\be
\label{eq:3terms}
I^N(\phi^{\cdot,N})=\frac12\int_0^T a_1-2a_2+a_3~ dt
\ee
where the three terms are given by
\beann
a^N_1&=&\langle (\phi_t^{\cdot,N})'+\alpha \phi_t^{\cdot,N},(\phi_t^{\cdot,N})'+\alpha \phi_t^{\cdot,N}\rangle_N,\\
a^N_2&=&\langle (\phi_t^{\cdot,N})'+\alpha \phi_t^{\cdot,N},  KF^{\cdot,N}(\phi_t^{\cdot,N})\rangle_N,\\
a^N_3&=&[\tilde{KF}^{\cdot,N}(\phi^{\cdot,N})]^T~ [\tilde{KF}^{\cdot,N}(\phi^{\cdot,N})]
\eeann
and $(\tilde{KF})^{i,N}=\frac{1}{\lambda_i^2}(KF)^{i,N}$.
\end{lemma}

\begin{proof}
For notational simplicity we shall temporarily omit in this proof the subscript for the inner product 
$\langle\cdot,\cdot\rangle_N=\langle\cdot,\cdot\rangle$ as well as the Galerkin index {e.g.}~$\phi_t^{\cdot,N}=\phi_t$
as it is understood that we work with $N$-dimensional vectors in this proof. 
Consider the following general calculation
\beann
\langle \phi_t'-g(\phi_t),\phi_t'-g(\phi_t)\rangle
&=&\langle \phi_t',\phi_t'\rangle -2 \langle \phi_t',g(\phi_t)\rangle + \langle g(\phi_t),g(\phi_t)\rangle \\
&=& \langle \phi_t',\phi_t'\rangle +2\alpha \langle \phi_t',\phi_t\rangle -2 \langle \phi_t',  KF(\phi_t)\rangle\\
&&+\langle KF(\phi_t),KF(\phi_t)\rangle + \alpha^2\langle \phi_t,\phi_t\rangle
-2\alpha \langle \phi_t,KF(\phi_t)\rangle\\
&=& \langle \phi_t'+\alpha \phi_t,\phi_t'+\alpha \phi_t\rangle -
2 \langle \phi_t'+\alpha \phi_t,  KF(\phi_t)\rangle +
 \tilde{KF}(\phi_t)^T~ \tilde{KF}(\phi_t).
\eeann
and observe that the result is independent of $N$.
\end{proof}

It is important to point out that the LDP from Theorem \ref{thm:SDE_LDP} requires the infimum
of the rate function. From Lemma \ref{lem:3terms} we know that the rate function splits into
three terms. The three terms are interesting in the asymptotic limit $N\ra \I$. Suppose 
\benn
(\phi_t^{\cdot,N})'+\alpha \phi_t^{\cdot,N}=\cO(\kappa(N))\qquad \text{and}
\qquad\tilde{KF}^{\cdot,N}(\phi_t^{\cdot,N})=\cO(\eta(N))
\eenn
as $N\ra\I$ for some non-negative functions $\kappa,\eta$. Then Lemma \ref{lem:3terms} yields
\benn
a^N_1=\cO(\kappa(N)^2\lambda_N^{-2}),\qquad a^N_2=\cO(\kappa(N)\eta(N)\lambda_N^{-1}),
\qquad a^N_3=\cO(\eta(N)^2).
\eenn

\begin{lemma}
\label{lem:fbounded}
Suppose there exists a positive constant $K_f$ such that 
\be
\label{lem:assume_f}
\sup_{x\in\R} |f(x)|\leq K_f
\ee
then $\eta(N)=1$.
\end{lemma}

\begin{proof}
A direct estimate yields
\benn
\left|\tilde{KF}^{j,N}(\phi_t^{\cdot,N})\right|\leq \int_\dom \left|f(\sum_{i=1}^N \phi_t^{i,N}v_i(x))\right||v_j(x)|dx
\leq K_f \int_\dom |v_j(x)|dx. 
\eenn
Since $\|v_j\|_{L^2(\dom)}=1$ and $L^2(\dom)\hookrightarrow L^1(\dom)$ the last integral is
uniformly bounded over $j\in\N$ by $\text{meas}(\dom)^{1/2}$. 
\end{proof}

We remark that several typical functions $f$ discussed in Section \ref{sec:Amari} such as $f(u)=(1+e^{-u})^{-1}$ and  
$f(u)=(\tanh(u)+1)/2$ are globally bounded so that Lemma \ref{lem:fbounded} does apply to many practical cases.
In this situation we get that
\benn
a^N_1=\cO(\kappa(N)^2\lambda_N^{-2}),\qquad a^N_2=\cO(\kappa(N)\lambda_N^{-1}),
\qquad a^N_3=\cO(1).
\eenn
We make a case distinction between the different relative asymptotics of $\kappa(N)$ and $\lambda_N$. Note that
the following asymptotic relations are purely formal:

\begin{itemize}
 \item If $\kappa(N)\ll \lambda_N$ or $\kappa(N)\sim \lambda_N$ as $N\ra \I$ then 
 we can conclude that $\kappa(N)\ra 0$ {i.e.} 
\be
\label{eq:exp_decay}
(\phi_t^{N,N})'+\alpha \phi_t^{N,N}\ra 0\qquad \text{as $N\ra 0$}
\ee
since for trace-class noise we know that $\lambda_N\ra 0$. If we formally require that
$(\phi_t^{N,N})'+\alpha \phi_t^{N,N}=0$ for $N$ sufficiently large then the higher-order
Galerkin modes decays exponentially in time
\benn
\phi_t^{N,N}=\phi_0^{N,N}e^{-\alpha t}.
\eenn
 \item If $\kappa(N)\gg \lambda_N$ as $N\ra \I$ then $a_1\gg -2a_2+a_3$ and the
first term dominates the asymptotics. But $a_1^N\geq 0$ for all $N$ so that the rate 
function only has a finite infimum if $a_1^N\ra 0$ as $N\ra \I$. This implies again that
\eqref{eq:exp_decay} holds for the case of a finite infimum.
\end{itemize}

Hence we get in many reasonable first-exit problems for the Amari model with trace-class
noise that there is a finite
set for $n\leq N$ of `slow' or `center-like' directions and an infinite set of `fast'
or `stable' directions for $n>N$. Although we have made this observation from the rate
function alone, it is entirely natural considering the structure of the Galerkin 
approximation. Indeed, for the case when the eigenvalues
of $K$ and $Q$ are related we may write \eqref{eq:galerkin_approx_system} as
\be
\label{eq:Galer_simple}
du^{i,N}_t=\left(-\alpha u^{i,N}_t+\lambda_i[\cdots]\right)dt+\epsilon \lambda_i d\beta_t^i
\ee
so that for bounded nonlinearity $f$, which is represented in the terms $[\cdots]$ in
\eqref{eq:Galer_simple}, the higher-order modes should really just be governed by 
$du^{i,N}_t=-\alpha u^{i,N}_t~dt$.\medskip

Hence, Propositions \ref{prop:our_fd_LDP}-\ref{prop:Gal_LDP_conv} and the multi-scale
nature of the problem induced by the trace-class noise suggest a procedure how to 
approximate the rate function and the associated LDP in practice. In particular, we
may compute the eigenvalues and eigenfunctions of $K$ and $Q$ up to a sufficiently large 
given order $N^*$. This yields an explicit representation of the Galerkin system and
the associated rate function. Then one may apply any finite-dimensional technique to understand
the rate function. One may even find a better truncation order $N< N^*$ based on the
knowledge that the minimizer of the rate function must have components that decay (almost)
exponentially in time for orders bigger than $N$.

\section{Outlook}
\label{sec:discussion}

In this paper we have discussed several steps towards a better understanding of noise-induced transitions in
continuum neural fields. Although we have provided the main basic elements via the LDP and finite-dimensional
approximations, there are still several very interesting open problems.\medskip

We have demonstrated that a sharp Kramers' rate calculation for neural fields with trace-class noise is
very challenging as the techniques for white-noise gradient-structure SPDEs cannot be applied directly. However,
we have seen in Section \ref{sec:deterministic} that the deterministic dynamics for neural fields frequently
exhibits a classical bistable structure with a saddle-state between stable equilibria. This suggests that
there should be a Kramers' law with exponential scaling in the noise intensity as well as a precisely
computable pre-factor. It is interesting to ask how this pre-factor depends on the eigenvalues of the trace-class operator $Q$
definining the $Q$-Wiener process. We expect that new technical tools are needed to answer this question.\medskip 

From the viewpoint of experimental data the exponential scaling for the LDP is relevant as it shows that 
noise-induced transitions have exponential interarrival times. This leads to the possibility that working memory
as well as perceptual bistability could be governed by a Poisson process. However, the same phenomena could also
be governed by a slowly varying variable {i.e.}~by an adaptive neural field \cite{BressloffReview}; 
the `fast' activity variable $U$ in the Amari model is augmented by one or more `slow' variables. In this context,
the required assumptions on the equilibrium structure in Section \ref{sec:deterministic} and the noise in Section \ref{sec:gain_perturb}
are not necessary to produce a bistable switch and the fast variable $U$ can, {e.g.}, just have a single deterministically 
unstable equilibrium and bistable, non-random switching between metastable states may occur. Of course, there is also the  
possibility that an intermediate regime between noise-induced and deterministic escape is relevant \cite{MeiselKuehn}.

It is interesting to note that the same problem arises generically across many natural sciences in the 
study of critical transitions (or `tipping points') \cite{Schefferetal,KuehnCT2}. The question which escape mechanism
from a metastable state matches the data is often discussed very controversially and we shall not aim to provide a
discussion here. However, our main goal to make the LDP and its associated rate functional as explicit as possible
should definitely help to simplify comparison between models and experiment. For example, a parameter study or data assimilation 
for the finite-dimensional Galerkin system considered in Theorem \ref{thm:galerkin_convergence} and the associated 
rate function in Proposition \ref{prop:our_fd_LDP} are often easier than 
working directly with the abstract solutions of the stochastic Amari model in $C([0,T],L^2(\dom))$.
\medskip

To study the parameter dependence is an interesting open question which we aim to address in future work. In particular, the next step 
is to use the Galerkin approximations in Section \ref{sec:Galerkin} and the associated LDP in Section \ref{sec:approximation} for
numerical purposes \cite{KuehnRiedler1}. Recent work for SPDEs \cite{BloemkerJentzen} suggests that a
spectral method can also be efficient for stochastic neural fields. Results on numerical continuation and jump
heights for SDEs \cite{KuehnSDEcont1} can also be immediately transferred to the spectral approximation
which would allow for studies of bifurcations and associated noise-induced phenomena.\medskip

One may also ask how far the technical assumptions we make in this paper can be weakened. It is not clear which
parts of the global Lipschitz assumptions may be replaced by local assumptions or removed altogether. Similar
remarks apply to the multiscale nature of the problem induced by the decay of the eigenvalues of $Q$. How far
this observation can be exploited to derive more efficient analytical as well as numerical techniques remains
to be investigated.\medskip

On a more abstract level it would certainly be desirable to extend our basic framework to other topics that 
have been considered already for deterministic neural fields. A generalization
to activity based models with nonlinearity $f(\int_\cB w(x,y)u(y)dy)$ seems possible. Furthermore, it
may be highly desirable to go beyond stationary solutions and investigate noise-induced 
switching and transitions for travelling waves and patterns.\medskip

\textbf{Acknowledgements:} CK would like to thank the European Commission (EC/REA) 
for support by a Marie-Curie International Re-integration Grant and the Austrian Academy of
Sciences (\"{O}AW) for support via an APART fellowship. We also would like to thank two 
anonymous referees whose comments helped to improve the manuscript.

\medskip

\textbf{Authors' contributions:} Both authors contributed equally to the paper.

\medskip

\textbf{Competing interests:} The authors declare that they have no competing interests.

\begin{appendix}
\section{Convergence of the Galerkin Approximation}

\begin{proof}[Proof of Theorem \ref{thm:galerkin_convergence}]
We fix a $T>0$. Throughout the proof an unspecified norm $\|\cdot\|$ or operator 
norm $\opnorm\cdot\opnorm$, respectively, are either for the Hilbert space $L^2(\dom)$ 
or the Banach space $C(\dom)$ and estimates using the unspecified notation are valid in 
both cases. Furthermore, $C>0$ denotes an arbitrary deterministic constant which may 
change from line to line but depend only on $T$. We begin the proof obtaining an a-priori 
growth bound on the solution of the Amari equation \eqref{eq:Amari2}. Using the linear 
growth condition on $F$ implied by its Lipschitz continuity we obtain the estimate
\benn
\|U_t\|\leq \e^{-\alpha t}\|U_0\| + C\int_0^t \e^{-\alpha(t-s)} (1+\|U_s\|)\,ds+\|O_t\|\,.
\eenn
Due to Gronwall's inequality there exists a deterministic constant $C$ such that it holds almost surely
\be
\label{eq:apriori_growth_bound}
\sup_{t\in[0,T]}\|U_t\|\leq C\Bigl(1+\|U_0\|+\sup_{t\in[0,T]}\|O_t\|\Bigr)\,\e^{CT}\qquad\textnormal{a.s.}
\ee
Note that $O$ is an Ornstein-Uhlenbeck process and it thus holds $\sup_{t\in[0,T]}\|O_t\|_{L^2}<\I$ 
almost surely and under the assumptions of Lemma \ref{lemma:spatial_continuity} in addition $
\sup_{t\in[0,T]}\|O_t\|_0<\I$ almost surely.

Let $P^N$ denote the projection operator from $L^2(\dom)$ to the subspace spanned by the 
first $N$ basis functions. Then we find that in Hilbert space notation the $N$-th Galerkin 
approximation satisfies
\benn
U^N_t= \e^{-\alpha t}\,P^N U_0 +\int_0^t \e^{-\alpha(t-s)} P^NKF(U^N_t)\,ds+\epsilon\, O^N_t\,.
\eenn
Here we use $O^N$ to be shorthand for the truncated stochastic convolution
\be\label{eq:OU_series}
O^N_t := \sum_{i=1}^N \lambda_i\int_0^t \e^{-\alpha (t-s)}\,d\beta^i_s ~v_i\,.
\ee
Hence we obtain for the error of the Galerkin approximation
\beann
U_t-U^N_t &=& \e^{-\alpha t}(U_0-P^NU_0) + \int_0^t \e^{-\alpha(t-s)} 
\Bigl(KF(U_t)-P^NKF(U^N_t)\Bigr)\,ds\\[1ex]
&&\mbox +\epsilon\,\bigl(O_t-O^N_t\bigr)\,.
\eeann
Adding and subtracting the obvious terms yields for the norm the estimate
\beann
\|U_t-U^N_t\| &\leq& \e^{-\alpha t}\|U_0-P^NU_0\| + 
\opnorm P^N K\opnorm \int_0^t \e^{-\alpha(t-s)} \|F(U_s)-F(U^N_s)\|\,ds\\[1ex]
&&\mbox{} + \opnorm K-P^NK\opnorm\,\int_0^t \e^{-\alpha(t-s)}\|F(U_s)\|\,ds +
\epsilon\,\|O_t-O^N_t\|\,,
\eeann
where $\opnorm P^N K\opnorm_{L^2}\leq\opnorm K\opnorm_{L^2}$ and 
$\sup_{N\in\N}\opnorm P^N K \opnorm_0<\I$ as a consequence of \cite[Lemma 11.1.4]{AtkinsonHan} 
(cf.~the application of this result below). Next, using the Lipschitz and linear growth 
conditions on $F$, applying Gronwall's inequality, taking the supremum over all $t\in[0,T]$ 
and estimating using the bound \eqref{eq:apriori_growth_bound} yield
\bea
\sup_{t\in[0,T]}\|U_t-U^N_t\|&\leq&C\left(\|U_0-P^N U_0\|+
\opnorm K-P^NK\opnorm\,\Bigl(1+\|U_0\|+\sup_{t\in[0,T]}\|O_t\|\Bigr)\right)\nonumber\\[1ex]
&&\mbox{} + C\left(\sup_{t\in[0,T]}\|O_t-O^N_t\|\right).
\eea
It remains to show that the individual terms in the right hand side converge to zero 
for $N\to\I$ almost surely.
\begin{itemize}
\item It clearly holds that $\|U_0-P^NU_0\|_{L^2}\to 0$ and the convergence $\|U_0-P^NU_0\|_0\to 0$ 
 holds by assumption.
\item Next, as argued above $\Bigl(1+\|U_0\|+\sup_{t\in[0,T]}\|O_t\|\Bigr)$ is a.s.~finite 
and the compactness of the operator $K$ implies $\opnorm K-P^NK\opnorm\to 0$ for $N\to\infty$, 
see \cite[Lemma 12.1.4]{AtkinsonHan}.
\item Finally, the third error term $\sup_{t\in[0,T]}\|O_t-O^N_t\|$ vanishes if the Galerkin 
approximations $O^N$ of the Ornstein-Uhlenbeck process $O$ converge almost surely in the 
spaces $C([0,T],L^2(\dom))$ and $C([0,T],C(\dom))$, respectively. This convergence is proven 
in Lemma \ref{res:conv_of_noise_approx} below.
\end{itemize}
The proof is completed.
\end{proof}

The following lemma contains the convergence of the Galerkin approximation of the 
Ornstein-Uhlenbeck process necessary for proving Theorem \ref{thm:galerkin_convergence}.

\begin{lemma}
\label{res:conv_of_noise_approx} There exists a sequence $b_N>0$ with $\lim_{N\to\I} b_N=0$ 
such that for all $T>0$ and all $\delta>0$ there exists a random variable $Z_\delta$ 
with $\E|Z_\delta|^p<\I$ for all $p\geq 1$ such that
\benn
\sup_{t\in[0,T]}\|O_t-O^N_t\|_{L^2}\,\leq\, Z_\delta\,b_N^{1-\delta}
\eenn
almost surely. If, in addition, the series $\sum_{i=1}^\I \lambda_i^2v_i^2$ 
converges in $C(\dom)$ and the functions $v_i$ are Lipschitz continuous with 
Lipschitz constants $L_i$ such that $\sup_x{\in\dom}\sum_{i=1}^\I 
\lambda_i^2 L_i^{2\rho}|v_i(x)|^{2(1-\rho)}<\I$ for a $\rho\in(0,1)$, 
then it further holds that
\benn
\sup_{t\in[0,T]}\|O_t-O^N_t\|_{0}\,\leq\, Z_\delta\,b_N^{1-\delta}
\eenn
almost surely.
\end{lemma}

\begin{remark} 
Assumptions on the speed of convergence of the series 
$\sum_{i=1}^\I \lambda_i^2$ and $\sum_{i=1}^\I \lambda_i^2v_i^2$ and 
$\sup_x{\in\dom}\sum_{i=1}^\I 
\lambda_i^2 L_i^{2\rho}|v_i(x)|^{2(1-\rho)}$ readily yield a rate of convergence 
for the Galerkin approximation due to the definition of the constants $b_N$ 
in the proof of the lemma.
\end{remark}

\begin{proof} As in the proof Theorem \ref{thm:galerkin_convergence} the unspecified 
norm $\|\cdot\|$ denotes either the norm in $L^2(\dom)$ or in $C(\dom)$ and estimates 
are valid in both cases. We fix $T>0$, $\rho\in(0,1)$ and a $p\in\N$ with $p>2d/\rho$. 
Throughout the proof $C>0$ denotes a constant that changes from line to line but depends 
only on the fixed parameters $T,p,\rho,\alpha$ and the domain $\dom\subset\R^d$.

Then we obtain for all $N,M\in\N$ with $M<N$ using the factorization method 
(cf.~\cite[Sec.~5.3]{daPratoZabczyk}) similarly to the proof of \cite[Lemma 5.6]{BloemkerJentzen} 
the estimate
\benn
\bigl(\E\sup\nolimits_{t\in[0,T]}\|O^N_t-O^M_t\|^p\bigl)^{1/p}\,\leq
\,C\,\sup\nolimits_{t\in[0,T]}\bigl(\E\|Y^{M,N}_t\|^p\bigr)^{1/p}\,,
\eenn
where $Y^{N,M}_t$ is the process defined by
\benn
Y^{M,N}_t=\sum_{i=M+1}^N \lambda_i\int_0^t(t-s)^{-\rho/2} \e^{-\alpha(t-s)}\,d\beta^i_s ~v_i\,.
\eenn
In order to estimate the $p$-th mean of the process $Y^{M,N}$ we proceed separately 
for the two cases $L^2(\dom)$ and $C(\dom)$.\medskip

\emph{The case of $L^2(\dom)$}: Due to the orthogonality of the basis functions and 
employing H\"older's inequality one obtains
\beann
\E\|Y^{M,N}_t\|^p_{L^2}&=&\E\left(\sum_{i=M+1}^N \lambda_i^2\,
\Bigl(\int_0^t(t-s)^{-\rho/2} \e^{-\alpha(t-s)}\,d\beta^i_s\Bigr)^2\right)^{p/2}\\[1ex]
&=&\E\left(\sum_{i=M+1}^N \lambda_i^{\frac{2(p-2)}{p}}\,
\Bigl(\lambda_i^{2/p}\int_0^t(t-s)^{-\rho/2} \e^{-\alpha(t-s)}\,d\beta^i_s\Bigr)^2\right)^{p/2}\\
&\leq&\E\left(\left(\sum_{i=M+1}^N \lambda_i^2\right)^{\frac{p-2}{p}}\left(\sum_{i=M+1}^N\Bigl(\lambda_i^{2/p}\int_0^t(t-s)^{-\rho/2} \e^{-\alpha(t-s)}\,d\beta^i_s\Bigr)^p\right)^{2/p}\right)^{p/2}\\
&\leq&\left(\sum_{i=M+1}^N \lambda_i^2\right)^{\frac{p-2}{2}}
\sum_{i=M+1}^N\E\left(\lambda_i^{2/p}\int_0^t(t-s)^{-\rho/2} \e^{-\alpha(t-s)}\,d\beta^i_s\right)^p.
\eeann
Next, as the stochastic integrals in the right hand side are centered Gaussian 
random variables \cite[Lemma 5.2]{BloemkerJentzen}\footnote{For a centred Gaussian 
random variable $Z$ it holds $\E Z^p\leq p! (\E Z^2)^{p/2}$ for all $p\in\N$.} yields for all $t\leq T$
\beann
\E\|Y^{M,N}_t\|^p_{L^2}&\leq&C\,\left(\sum_{i=M+1}^N \lambda_i^2\right)^{\frac{p-2}{2}}
\sum_{i=M+1}^N \lambda_i^2\left(\int_0^t(t-s)^{-\rho} \e^{-2\alpha(t-s)}\,ds\right)^{p/2}\\[1ex]
&\leq& C\,\left(\sum_{i=M+1}^N \lambda_i^2\right)^{\frac{p-2}{2}}\sum_{i=M+1}^N \lambda_i^2\left(\int_0^T s^{-\rho} \e^{-2\alpha s}\,ds\right)^{p/2}\\[1ex]
&\leq& C\,\left(\sum_{i=M+1}^N \lambda_i^2\right)^{p/2}\,.
\eeann
Therefore we obtain for all $M,N\in\N$ with $M<N$
\be
\label{eq:Galerkin_conf_final_1}
\Bigl(\sup\nolimits_{t\in[0,T]}\E\|Y^{N,M}_t\|^p_{L^2}\Bigr)^{1/p}\,\leq
\, C\left(\sum_{i=M+1}^N \lambda_i^2\right)^{1/2}\,\leq\, C\left(\sum_{i=M+1}^\I \lambda_i^2\right)^{1/2},
\ee
where the final upper bound decreases to zero for $M\to\I$ by assumption. \medskip

\emph{The case of $C(\dom)$}: In this case the estimates get a bit more involved. As 
$\rho/2>d/p$ The continuous embedding of the Sobolev-Slobodeckij space $W^{\rho/2,p}(\dom)$ 
into $C(\dom)$ (cf.~\cite[Sec.~2.2.4 and 2.4.4]{RunstSickel}) and 
\cite[Lemma 5.2]{BloemkerJentzen} yield the estimates
\bea
\label{eq:proof_galerkin_lemma}
\sup\nolimits_{t\in[0,T]}\E\|Y^{N,M}_t\|_0^p&\leq& C \sup\nolimits_{t\in[0,T]} 
\int_\dom\int_\dom \frac{\E |Y^{M,N}_t(x)-Y^{M,N}_t(y)|^p}{|x-y|^{d+\rho p/2}} dx ~dy\nonumber\\[1ex]
&&\mbox{} + C \sup\nolimits_{t\in[0,T]} \int_\dom \E|Y^{M,N}(x)|^p dx\nonumber\\[1ex]
&\leq& C \sup\nolimits_{t\in[0,T]} \int_\dom\int_\dom 
\frac{\bigl(\E |Y^{M,N}_t(x)-Y^{M,N}_t(y)|^2\bigr)^{p/2}}{|x-y|^{d+\rho p/2}} dx ~dy\nonumber\\[1ex]
&&\mbox{} + C \sup\nolimits_{t\in[0,T]} \int_\dom \bigl(\E|Y^{M,N}(x)|^2\Bigr)^{p/2} dx\,.
\eea
We proceed estimating the two expectation terms in the right hand side. Then we 
obtain for all $M<N$ and all $x,y\in\dom$ for the first term
\bea
\label{eq:Galerkin_conv_second_estimate}
\E |Y^{M,N}_t(x)-Y^{M,N}_t(y)|^2&=&\E \Big|\sum_{i=M+1}^N \lambda_i\int_0^t(t-s)^{-\rho/2} 
\e^{-\alpha(t-s)}\,d\beta^i_s ~\bigl(v_i(x)-v_i(y)\bigr)\Big|^2\nonumber\\[1ex]
&\leq&\sum_{i=M+1}^N \lambda_i^2 \int_0^T s^{-\rho}\e^{-2\alpha s}\, ds ~|v_i(x)-v_i(y)|^2\nonumber\\[1ex]
&\leq& C \sum_{i=M+1}^N \lambda_i^2L_i^{2\rho}\,|x-y|^{2\rho}
\eea
for any $\rho\in(0,1)$ and for the second term
\be
\label{eq:Galerkin_conv_first_estimate}
\E|Y^{M,N}_t(x)|^2\,\leq\,\sum_{i=M+1}^N \lambda_i^2\int_0^t(t-s)^{-\rho} 
\e^{-2\alpha(t-s)}\,ds\,v_i(x)^2\,\leq\, C\sum_{i=M+1}^N \lambda_i^2v_i(x)^2\,.
\ee
Next applying the estimates \eqref{eq:Galerkin_conv_first_estimate} and 
\eqref{eq:Galerkin_conv_second_estimate} to the right hand side of 
\eqref{eq:proof_galerkin_lemma} yields, note that $\rho p/2-d>0$,
\beann
\lefteqn{\Bigl(\sup\nolimits_{t\in[0,T]}\E\|Y^{N,M}_t\|_0^p\Bigr)^{1/p}}\\[1ex]
&\leq& C \Bigl(\int_\dom\int_\dom \frac{\Bigl(\sum_{i=M+1}^N \lambda_i^2 
L_i^{2\rho} ~|x-y|^{2\rho}\Bigr)^{p/2}}{|x-y|^{d+\rho p/2}} dx ~dy + 
\int_\dom \Bigl(\sum_{i=M+1}^N \lambda_i^2 v_i(x)^2\Bigr)^{p/2} dx\Bigr)^{1/p}\\[1ex]
&\leq& C\biggl(\int_\dom\int_\dom |x-y|^{\rho p/2-d}dx ~dy\, \Bigl(\sum_{i=M+1}^N 
\lambda_i^2 L_i^{2\rho}\Bigr)^{p/2} + \Bigl(\sup_{x\in \dom}\Bigl|\sum_{i=M+1}^N 
\lambda_i^2v_i(x)^2\Bigr| \Bigr)^{p/2}  \biggr)^{1/p}\\[1ex]
&\leq& C\biggl(\sup_{x\in \dom}\Bigl|\sum_{i=M+1}^N \lambda_i^2v_i(x)^2\Bigr| 
+\sum_{i=M+1}^N \lambda_i^2L_i^{2\rho} \biggr)^{1/2}
\eeann
for any $\rho\in(0,1)$. Due to the assumptions of the lemma the two summations 
in the right hand side converge for $N\to\I$ and thus we obtain for all $M,N\in\N$ 
with $M<N$ the estimate
\be
\label{eq:Galerkin_conf_final_2}
\Bigl(\sup\nolimits_{t\in[0,T]}\E\|Y^{N,M}_t\|_0^p\Bigr)^{1/p}\,\leq\, 
C\,\biggl(\sup_{x\in \dom}\Bigl|\sum_{i=M+1}^\I \lambda_i^2v_i(x)^2\Bigr| 
+\sum_{i=M+1}^\I \lambda_i^2L_i^{2\rho} \biggr)^{1/2},
\ee
where the right hand side decreases to zero for $M\to\I$. \medskip

Overall, we infer from the estimates \eqref{eq:Galerkin_conf_final_1} and 
\eqref{eq:Galerkin_conf_final_2} that $O^N$ is a Cauchy-sequence in the two 
spaces $C([0,T],L^2(\dom))$ and $C([0,T],C(\dom))$ with respect to convergence 
in the $p$-th mean and the limit is given by the process $O$. Moreover, it holds that
\be
\label{eq:Galerkin_conv_sup_bound}
\Bigl(\E\sup\nolimits_{t\in[0,T]}\|O_t-O^N_t\|^p\Bigr)^{1/p}\,\leq\, C\,b_N\, \qquad\forall\, N\in\N\,,
\ee
where the constant $C$ depends only on $p$ but is independent of $N$ and 
the sequence $b_N$ is independent of $p$ and $\lim_{N\to \I} b_N=0$. As we 
fixed $p\in\N$ sufficiently large at the beginning of the proof, the result 
\eqref{eq:Galerkin_conv_sup_bound} holds for all sufficiently large $p\in\N$. 
Then, however, Jensen's inequality implies that \eqref{eq:Galerkin_conv_sup_bound} 
holds for all $p\in[1,\infty)$. Proceeding as in the proof of \cite[Lemma 2.1]{KloedenNeuenkirch} 
using the Chebyshev-Markov inequality and the Borel-Cantelli lemma one obtains 
that there exists for all $\delta>0$ a random variable $Z_\delta$ with $\E|Z_\delta|^p<\I$ 
for all $p\geq 1$ such that
\be
\sup\nolimits_{t\in[0,T]}\|O_t-O^N_t\|\,\leq\, Z_\delta\,b_N^{1-\delta}\,\qquad\textnormal{almost surely}.
\ee
The proof is completed.
\end{proof}
\end{appendix}

\bibliographystyle{plain}
\bibliography{./KR}

\end{document}